\documentclass[a4paper, reqno, 12pt]{amsart}

\usepackage[usenames,dvipsnames]{color}
\usepackage{amsthm,amsfonts,amssymb,amsmath,amsxtra}
\usepackage[all]{xy}
\SelectTips{cm}{}
\usepackage{xr-hyper}
\usepackage[colorlinks=
   citecolor=Black,
   linkcolor=Red,
   urlcolor=Blue]{hyperref}
\usepackage{verbatim}
\usepackage{enumitem}
\usepackage{tikz,tikz-cd}
\usepackage{graphicx}
\usepackage{adjustbox}
\usetikzlibrary{positioning}

\usepackage[margin=1.25in]{geometry}
\usepackage{mathrsfs}

\newlist{spadelist}{enumerate}{1}
\setlist[spadelist]{
  label=$\spadesuit$\ \arabic*.,  % 注意：添加 $ $ 数学环境
  ref=$\spadesuit$\ \arabic*      % 引用时同样需要数学环境
}

\RequirePackage{xspace}
\RequirePackage{etoolbox}
\RequirePackage{varwidth}
\RequirePackage{enumitem}
\RequirePackage{tensor}
\RequirePackage{mathtools}
\RequirePackage{longtable}
\RequirePackage{multirow}

\def\<{\langle}
\def\>{\rangle}

\newcommand{\fka}{\ensuremath{\mathfrak{a}}\xspace}
\newcommand{\fkb}{\ensuremath{\mathfrak{b}}\xspace}

\newcommand{\fkg}{\ensuremath{\mathfrak{g}}\xspace}

\newcommand{\fkl}{\ensuremath{\mathfrak{l}}\xspace}
\newcommand{\fkm}{\ensuremath{\mathfrak{m}}\xspace}
\newcommand{\fkn}{\ensuremath{\mathfrak{n}}\xspace}

\newcommand{\fkp}{\ensuremath{\mathfrak{p}}\xspace}
\newcommand{\fkq}{\ensuremath{\mathfrak{q}}\xspace}

\newcommand{\fkt}{\ensuremath{\mathfrak{t}}\xspace}
\newcommand{\fku}{\ensuremath{\mathfrak{u}}\xspace}
\newcommand{\fkv}{\ensuremath{\mathfrak{v}}\xspace}

\newcommand{\fkz}{\ensuremath{\mathfrak{z}}\xspace}

\newcommand{\BC}{\ensuremath{\mathbb {C}}\xspace}

\newcommand{{\BG}}{\ensuremath{\mathbb {G}}\xspace}

\newcommand{{\BK}}{\ensuremath{\mathbb {K}}\xspace}

\newcommand{\BR}{\ensuremath{\mathbb {R}}\xspace}

\newcommand{\BZ}{\ensuremath{\mathbb {Z}}\xspace}

\newcommand{\CA}{\ensuremath{\mathcal {A}}\xspace}
\newcommand{\CB}{\ensuremath{\mathcal {B}}\xspace}
\newcommand{\CC}{\ensuremath{\mathcal {C}}\xspace}
\newcommand{\CD}{\ensuremath{\mathcal {D}}\xspace}
\newcommand{\CE}{\ensuremath{\mathcal {E}}\xspace}
\newcommand{\CF}{\ensuremath{\mathcal {F}}\xspace}

\newcommand{\CJ}{\ensuremath{\mathcal {J}}\xspace}

\newcommand{\CL}{\ensuremath{\mathcal {L}}\xspace}
\newcommand{\CM}{\ensuremath{\mathcal {M}}\xspace}
\newcommand{\CN}{\ensuremath{\mathcal {N}}\xspace}
\newcommand{\CO}{\ensuremath{\mathcal {O}}\xspace}

\newcommand{\CQ}{\ensuremath{\mathcal {Q}}\xspace}

\newcommand{\CS}{\ensuremath{\mathcal {S}}\xspace}
\newcommand{\CT}{\ensuremath{\mathcal {T}}\xspace}

\newcommand{\CV}{\ensuremath{\mathcal {V}}\xspace}

\newcommand{\CZ}{\ensuremath{\mathcal {Z}}\xspace}

\newcommand{\Ad}{{\mathrm{Ad}}}

\DeclareMathOperator{\Coker}{Coker}

\DeclareMathOperator{\End}{End}

\newcommand{\GL}{\mathrm{GL}}

\DeclareMathOperator{\Hom}{Hom}

\newcommand{\id}{\ensuremath{\mathrm{id}}\xspace}

\newcommand{\Ind}{{\mathrm{Ind}}}

\DeclareMathOperator{\Ker}{Ker}

\DeclareMathOperator{\Lie}{Lie}

\newcommand{\Rep}{{\mathrm{Rep}}}

\newcommand{\HC}{{\mathcal{HC}}}

\DeclareMathOperator{\Sym}{Sym}

\newcommand{\U}{\mathrm{U}}

\newcommand{\dslash}{/\!\!/}

\newcommand{\wt}{\mathrm{wt}}
\newcommand{\Bil}{\mathrm{Bil}}

\newcommand{\ov}{\overline}

\newcommand{\lra}{\longrightarrow}

\newcommand{\rmh}{\mathrm{H}}
\newcommand{\Tot}{\mathrm{Tot}}

\newcommand{\mind}{\bar{\times}}

\newcommand{\Span}{{\operatorname{Span}}}

\newcommand{\SInd}{\mathcal{S}\mathrm{Ind}}
\newcommand{\Smod}{\mathcal{S}mod}

\DeclareMathOperator{\supp}{supp}
\newtheorem{theorem}{Theorem}
\newtheorem{proposition}[theorem]{Proposition}
\newtheorem{lemma}[theorem]{Lemma}
\newtheorem {conjecture}[theorem]{Conjecture}
\newtheorem{corollary}[theorem]{Corollary}

\theoremstyle{definition}
\newtheorem{definition}[theorem]{Definition}
\newtheorem{example}[theorem]{Example}

\newtheorem{remark}[theorem]{Remark}

\numberwithin{equation}{section}
\numberwithin{theorem}{section}

\setitemize[0]{leftmargin=*,itemsep=\the\smallskipamount}
\setenumerate[0]{leftmargin=*,itemsep=\the\smallskipamount}

\renewcommand{\to}{%
   \ifbool{@display}{\longrightarrow}{\rightarrow}%
   }
\let\shortmapsto\mapsto
\renewcommand{\mapsto}{%
   \ifbool{@display}{\longmapsto}{\shortmapsto}%
   }
\newlength{\olen}
\newlength{\ulen}
\newlength{\xlen}
\newcommand{\xra}[2][]{%
   \ifbool{@display}%
      {\settowidth{\olen}{$\overset{#2}{\longrightarrow}$}%
       \settowidth{\ulen}{$\underset{#1}{\longrightarrow}$}%
       \settowidth{\xlen}{$\xrightarrow[#1]{#2}$}%
       \ifdimgreater{\olen}{\xlen}%
          {\underset{#1}{\overset{#2}{\longrightarrow}}}%
          {\ifdimgreater{\ulen}{\xlen}%
             {\underset{#1}{\overset{#2}{\longrightarrow}}}
             {\xrightarrow[#1]{#2}}}}%
      {\xrightarrow[#1]{#2}}
   }
\makeatother
\newcommand{\xyra}[2][]{%
   \settowidth{\xlen}{$\xrightarrow[#1]{#2}$}%
   \ifbool{@display}%
      {\settowidth{\olen}{$\overset{#2}{\longrightarrow}$}%
       \settowidth{\ulen}{$\underset{#1}{\longrightarrow}$}%
       \ifdimgreater{\olen}{\xlen}%
          {\mathrel{\xymatrix@M=.12ex@C=3.2ex{\ar[r]^-{#2}_-{#1} &}}}%
          {\ifdimgreater{\ulen}{\xlen}%
             {\mathrel{\xymatrix@M=.12ex@C=3.2ex{\ar[r]^-{#2}_-{#1} &}}}
             {\mathrel{\xymatrix@M=.12ex@C=\the\xlen{\ar[r]^-{#2}_-{#1} &}}}}}%
      {\mathrel{\xymatrix@M=.12ex@C=\the\xlen{\ar[r]^-{#2}_-{#1} &}}}%
   }
\makeatletter
\newcommand{\xla}[2][]{%
   \ifbool{@display}%
      {\settowidth{\olen}{$\overset{#2}{\longleftarrow}$}%
       \settowidth{\ulen}{$\underset{#1}{\longleftarrow}$}%
       \settowidth{\xlen}{$\xleftarrow[#1]{#2}$}%
       \ifdimgreater{\olen}{\xlen}%
          {\underset{#1}{\overset{#2}{\longleftarrow}}}%
          {\ifdimgreater{\ulen}{\xlen}%
             {\underset{#1}{\overset{#2}{\longleftarrow}}}
             {\xleftarrow[#1]{#2}}}}%
      {\xleftarrow[#1]{#2}}
   }
\newcommand{\isoarrow}{%
   \ifbool{@display}{\overset{\sim}{\longrightarrow}}{\xrightarrow\sim}%
   }
   
\begin{document}

\title[Archimedean Bernstein-Zelevinsky]{Canonical Bernstein-Zelevinsky Filtration and Casselman's Comparison Conjecture}
%\author[Kei Yuen Chan]{Kei Yuen Chan}
\author[Kaidi Wu]{Kaidi Wu}
\author[Jun Yu]{Jun Yu}

%\address{(Chan) Department of Mathematics, The University of Hong Kong, HK.}\email{kychan1@hku.hk}

\address{(Wu) Department of Mathematics and New Cornerstone Science Laboratory, The University of Hong Kong, HK.}
\email{kaidiwu24@connect.hku.hk}

\address{(Yu) School of Mathematical Sciences and Beijing International Center for Mathematical Research, Peking, 100871, China}
\email{junyu@bicmr.pku.edu.cn}

\thanks{}


\date{\today}

\begin{abstract}
    We establish a canonical Bernstein--Zelevinsky filtration for Casselman--Wallach representations that is analogous to the $p$-adic case. In addition, we outline an approach to Casselman's comparison conjecture and prove it for general linear groups, as well as for quasi-split even orthogonal groups in some special cases. We also give some applications of the Bernstein--Zelevinsky filtration, such as to the study of highest derivatives and the indecomposability of mirabolic restrictions.
\end{abstract}
\maketitle

\tableofcontents

\section{Introduction}
\subsection{} Let $G$ be a real reductive group and $P$ a parabolic subgroup of $G$ with unipotent radical $N$. A fundamental problem in the representation theory of $G$ is to understand the structures related to $P$, such as parabolic induction. Following~\cite{Fd,WZ}, two interrelated and fruitful perspectives emerge.

The first is \textit{spectral decomposition}. Namely, we aim to decompose the restriction of $\pi$ to $P$ according to irreducible unitary representations of $N$. The second is \textit{spectral reduction}: we study the coinvariants of $\pi$ with respect to irreducible unitary representations of $N$.

Analogous questions have been studied for $p$-adic groups since the pioneering works of Bernstein–Zelevinsky~\cite{BZ77} and Zelevinsky~\cite{Ze80}, which concern the spectral filtration of a smooth representation of $\GL_n(F)$ restricted to the mirabolic subgroup. (Here $F$ is a $p$-adic field.) The subquotients of this spectral filtration, also known as the Bernstein–Zelevinsky filtration, are described by compact induction from its derivatives. This filtration has been widely used in the investigation of various aspects of the Langlands program, such as branching laws (see \cite{CSa21,Ch25}) and local $L$-functions (see \cite{CPS17}).

On the other hand, two specific cases of spectral reduction also appear frequently in the literature. The first is the Jacquet functor
\[
\CJ_N(\pi) := \pi_N = \pi / \ov{\fkn \cdot \pi},
\]
which is the spectral reduction at the trivial character. One can also compare the homology of the Jacquet functor in the smooth setting and in the algebraic setting. This comparison is known as Casselman's comparison conjecture; see \cite[Conjecture 10.3]{Vog08}.

\begin{conjecture}\label{comparison conj}
    For any integer $i$, there is a natural isomorphism
    \[
    \rmh_i(\fkn, \pi_K)^{\infty} \simeq \rmh_i(\fkn, \pi),
    \]
    where $\pi_K$ is the Harish-Chandra module of $\pi$ consisting of $K$-finite vectors, and $(\cdot)^{\infty}$ denotes the Casselman–Wallach globalization.
\end{conjecture}

When $P$ is a minimal parabolic subgroup, the result was proved in~\cite{HT98} and~\cite{LLY21}. For the general case, it has long remained unresolved. Nevertheless, the study of branching laws and theta correspondence draws extensively on this result.

The second specific case is the generalized Whittaker model (see~\cite{GGS17}), which is closely connected with nilpotent invariants of representations. Some special cases in classical groups, called local descent in~\cite{JZ18}, form an important constituent of the local Gan–Gross–Prasad conjecture.

In this article, we study the spectral decomposition and spectral reduction through the coarse spectral filtration proposed in~\cite{WZ} when $N$ is abelian. As we mentioned in the introduction of~\cite{WZ}, the coarse spectral filtration is concretely constructed through detailed Fourier transform and non-canonical in nature. However, it contains enough information to deduce the canonical filtration and Casselman's comparison conjecture. We will introduce our main results in more detail in the following subsections.

\subsection{Spectral reduction}
For generality, we first assume that $P$ is an almost linear Nash group with a Levi decomposition $P = LN$, \textbf{where $N$ is abelian}. Let $\phi$ be a unitary character of $N$, and let $S_{\phi}$ be the stabilizer of $\phi$ in $P$. The (normalized) spectral reduction at $\phi$ is defined as the (normalized) twisted Jacquet functor
\[
 \Psi_{\phi}(\sigma):= \delta_{S_{\phi}\cap L} \otimes \sigma/ \Span\{\alpha v-\phi(\alpha)v\mid v\in\sigma, \alpha\in \fkn\},
\]
where $\delta_{S_{\phi}\cap L}$ is the modular character, and we regard $\phi$ as a character of the complexified Lie algebra $\fkn$ by derivation. Our first main result studies the reduction at open orbits.
\begin{theorem}[Proposition~\ref{open-iso}, Corollary~\ref{open-exa}]
    Let $\phi$ be a unitary character such that its $P$-orbit $\CO$ in $\widehat{N}$ is open. Then $\Psi_{\phi}(\sigma)$ is Hausdorff for any smooth moderate-growth Fr\'echet representation $\sigma$ of $P$. Moreover, if 
    \[
    0 \lra \sigma_1 \lra \sigma_2 \lra \sigma_3 \lra 0
    \]
    is a short exact sequence of such representations, and $\CS(\CO)\cdot \sigma_1$ is closed in $\sigma_1$, then
    \[
    0 \lra \Psi_{\phi}(\sigma_1) \lra \Psi_{\phi}(\sigma_2) \lra \Psi_{\phi}(\sigma_3) \lra 0
    \]
    is also exact.
\end{theorem}

The next key ingredient of this paper is to outline an approach to the comparison conjecture. Our central idea is to pass to a complete category and then establish a corresponding comparison between complete categories by comparing the standard objects within the complete category. Let us illustrate this in more detail.

Let $G$ be a real reductive group and $P$ a parabolic subgroup. In~\cite[Section 2.5]{WZ}, we introduced a smooth complete category $\CC(\fkg,L)_f$ with respect to the spectral reduction at the trivial character of $\fkn$. In Lemma~\ref{CW CJ in finite length cat}, we will prove that the smooth Casselman--Jacquet functor
\[
\mathrm{CJ}^{\infty}_N(\pi) := \varprojlim_k \, \pi / \ov{\fkn^k \pi}
\]
sends a Casselman-Wallach representation $\pi$ to an object in this complete category. In Section~\ref{alg complete cat sec}, we propose an algebraic complete category, denoted $\HC(\fkg,K_L)$, and prove that the algebraic Casselman-Jacquet functor
\[
\mathrm{CJ}_N(E) := \varprojlim_k \, E / \fkn^k E
\]
sends a Harish--Chandra module $E$ to an object in this algebraic complete category. 

The two complete categories share many similarities. For example, both have a family of standard objects, called \textit{formal Verma modules}, whose $\fkn$-homology is straightforward to compute, and every irreducible object arises as the unique irreducible quotient of some formal Verma module. Therefore, it is desirable to ask whether there is a category equivalence between these two categories.

In Section~\ref{alg complete cat sec}, we construct the category equivalence 
\[
\begin{tikzcd}[column sep=large]
\CC(\fkg,L)_f \arrow[r, bend left=25, "\cdot^{\flat}"] & \HC(\fkg,K_L) \arrow[l, bend left=25, "\widehat{\cdot}"]
\end{tikzcd}
\]
via the Casselman-Wallach globalization functor of the Levi subgroup $L$. Moreover, we demonstrate that the category equivalence is well-behaved with respect to standard objects. Therefore, we can prove the homology comparison between two complete categories (see Theorem~\ref{complet-com})
\[
\rmh_i(\fkn, E)^{\infty} \simeq \rmh_i(\fkn, \widehat{E})
\]
via replacing $E$ by standard objects. Here, $E\in \HC(\fkg,K_L)$ and $\cdot^{\infty}$ is the Casselman-Wallach globalization of a Harish-Chandra module of $L$.

On the other hand, the category equivalence also intertwines the Casselman-Jacquet functor. Namely, we have a natural isomorphism
\[
\widehat{ \mathrm{CJ}_N(E)} \simeq \mathrm{CJ}_N^{\infty}(E^{\infty}).
\]
The proof essentially uses the relationship between the smooth and algebraic generalized Jacquet functors established in~\cite[Theorem 12.1]{CWYZ}. This leads to the following theorem.   
\begin{theorem}[Proposition~\ref{HC-com}, Theorem~\ref{complet-com}]\label{general com}
    Let $E$ be a Harish-Chandra module of $G$. Then we have a natural isomorphism for any integer $i$
    \[
    \rmh_i(\fkn, E)^{\infty}\simeq \rmh_i(\fkn, \mathrm{CJ}_N^{\infty}(E^{\infty})).
    \]
\end{theorem}
The theorem reduces Casselman's comparison conjecture to studying the natural map for Casselman-Wallach representations $\pi$:
\begin{equation}\label{CW com eq}
    \rmh_i(\fkn,\pi) \lra \rmh_i(\fkn, \mathrm{CJ}_N^{\infty}(\pi)).
\end{equation}
This is where the coarse spectral filtration plays a fundamental role, see Definition~\ref{coarse def} for the definition of coarse spectral filtration. We remark that Lemma~\ref{reduce to max} reduces the comparison conjecture to maximal parabolic cases.
\begin{theorem}\label{CW com thm}
    Suppose that $P$ is maximal and any principal series of $G$ has a coarse spectral filtration. Then the natural map \eqref{CW com eq} is an isomorphism for any Casselman-Wallach representation of $G$.
\end{theorem}

Combining Theorem~\ref{general com} and Theorem~\ref{CW com thm} with~\cite[Theorem 3.11, Theorem 3.6]{WZ}, we prove the comparison conjecture~\ref{comparison conj} in the following specific cases.
\begin{theorem}
    The comparison conjecture holds for $G=\GL_n(\BR)$ and  $G=\GL_n(\BC)$, for arbitrary parabolic subgroups $P$. It also holds for quasi-split even orthogonal groups $\mathrm{O}_{2n}$ with respect to the parabolic subgroup $P_t$, whose Levi factor is isomorphic to
    $\underbrace{\GL_1\times\cdots\times\GL_1}_t\times  \mathrm{O}_{2n-2t}$.
\end{theorem}

\subsection{Spectral decomposition}
For a spectral decomposition of a Casselman-Wallach representation $\pi$ with respect to $P$, we mean a filtration of $\pi|_P$ such that each subquotient is a spectral component of $N$ described by a certain functor. Unlike the $p$-adic case, where the functor is simply given by the twisted Jacquet functor composed with compact induction, the co-normal derivative in the Archimedean case significantly complicates the issue. Fortunately, two spectra are accessible: the trivial spectrum and the spectrum corresponding to the open $P$-orbits in $\widehat{N}$. Roughly speaking, the trivial spectrum only requires the co-normal derivative, while the open spectrum involves only Schwartz induction. 

Moreover, since the topology matters in the Archimedean case, it is crucial that each submodule in the filtration of $\pi|_P$ be closed. To establish this, we utilize the rearrangement of the coarse spectral filtration introduced in~\cite{WZ}. Finally, we get the following result.
\begin{theorem}[Theorem~\ref{Canonical filtration}]\label{can-fil-intr}
   Suppose that any principal series of $G$ admits a coarse spectral filtration. Then, any Casselman-Wallach representation $\pi$ has a decreasing filtration as $P$-representations
    \begin{equation}
          \pi=\pi_0\supset \pi_1\supset \dots \supset \pi_r= 0 
    \end{equation}
    such that $\pi_{\ell}/\pi_{\ell+1}$ is an $N$-spectral component for any $\ell$. Moreover, we have functorial isomorphisms
    \begin{enumerate}
        \item $\pi_0/\pi_1\simeq \mathrm{CJ}^{\infty}_N(\pi)$;
        \item and $\pi_{\ell}/\pi_{\ell+1}\simeq \SInd_{S_{\phi}}^P(\Psi_{\phi}(\pi))$ if the $P$-orbit of $\phi\in\widehat{N}$ is open.
    \end{enumerate} 
\end{theorem}
When $G = \GL_n$ and $P = M_n$ is the mirabolic subgroup consisting of matrices whose last row is $(0, \dots, 0, 1)$, there are only the trivial spectrum and the open spectrum. Therefore, Theorem~\ref{can-fil-intr} yields a canonical Bernstein--Zelevinsky filtration analogous to the $p$-adic one in~\cite{BZ77}.
\begin{theorem}[Theorem~\ref{canon-fil-GL_n}]
    Let $\pi$ be a Casselman-Wallach representation of $\GL_n(\BR)$ or $\GL_n(\BC)$. Then $\pi|_{M_n}$ has a decreasing filtration
    \[
    \pi|_{M_n}=\pi_0\supset \pi_1\supset \dots\supset \pi_n= 0
    \]
    such that
    \[
    \pi_k/\pi_{k+1}\simeq I^{k} ( D^{k+1}(\pi)) \text{ for}\quad 0\leq k\leq n-1.
    \]
\end{theorem}
Here, the derivative functor $D^k$ involved in the description of the subquotient was originally introduced by~\cite{AGS15a,AGS15b} and is recalled in Definition~\ref{def-der-gl}.

\subsection{Applications to general linear groups}
In this subsection, we summarize some applications of the canonical Bernstein-Zelevinsky filtration to the study of general linear groups. First, we give an affirmative answer to the open question in~\cite[3.1, (3)]{AGS15a}.
\begin{proposition}[Corollary~\ref{derivative exact}]
        The derivative functor
    \[
    D^{k+1}:\mathrm{CW}_{\GL_n}\lra \Smod_{M_{n-k}}
    \]
    is an exact functor for any non-negative integer $k$.
\end{proposition}
Among all the derivatives, the highest derivative $\pi^-$ contains the most comprehensive information. For example, the computation of highest derivatives will reveal the irreducibility of monomial representations, see~\cite{G17}. In this article, we investigate the highest derivative of the contragredient representation through the bilinear pairing proposed by~\cite{BZ77}. 
\begin{theorem}[Corollary~\ref{pairing-coro}]
    Let $\pi$ be a Casselman-Wallach representation of $\GL_n$ of depth $d$. Then there is a non-degenerate $M_{n-d+1}$-equivariant bilinear map
    \[
    \pi^{-}\times  (\pi^{\vee})^- \lra \BC.
    \]
\end{theorem}
Unlike the situation for $p$-adic groups, where the highest derivative of an irreducible representation of depth $d$ is an irreducible representation of $\GL_{n-d}$, in the Archimedean case the highest derivative of an irreducible representation of $\GL_n$ is typically not irreducible. Thus, in subsection~\ref{conj sec}, we propose some conjectures about the structure of the highest derivatives (Conjecture~\ref{socle conjecture}) and give some convincing computations. In particular, We conjecture that, as a representation of $M_{n-d+1}$, it has a unique irreducible submodule. 

A primary motivation for studying the restriction to the mirabolic subgroup $M_n$ is that it serves as an intermediate step toward restriction to $\GL_{n-1}$. Although the main object of interest is the multiplicity space upon restriction to $\GL_{n-1}$, it is instrumental to study categorical properties, such as indecomposability or projectivity. For relevant research on $p$-adic groups, see~\cite{Ch21}. Some of these questions can first be asked for the mirabolic restriction. In subsection~\ref{inde sec}, we obtain the following result.
\begin{theorem}
    When $\pi$ is an irreducible unitary representation or an irreducible generic representation of $\GL_n$, it is homogeneous. In particular, it is indecomposable as a representation of $M_n$.
\end{theorem}
For the definition of homogeneous representations, readers are referred to Definition~\ref{homo def}. Our strategy is inspired by the argument in~\cite[section 6]{Ze80}. However, it is worth mentioning that, unlike in the $p$-adic case where the Zelevinsky classification is available, we do not know how to embed an irreducible representation into a well-behaved parabolic induced representation (such as a degenerate principal series) in such a way that both have the same depth. Thus, our proof has intrinsic difficulties that prevent it from generalizing to prove that any irreducible representation is homogeneous, though we conjecture that this is true. 

Proving this conjecture is desirable because, together with the conjecture that $\pi^-$ has a unique irreducible submodule, it would imply that every irreducible representation, when restricted to the mirabolic subgroup, is indecomposable.

\subsection{Computation of the Casselman-Jacquet functor} The last part of this article is devoted to preparing the computation of the Casselman--Jacquet functor for parabolically induced representations in~\cite{WZ26}. In~\cite{WZ26}, we will compute the Casselman--Jacquet functor using Mackey theory, which yields a filtration of the Casselman--Jacquet functor whose subquotients are formal Verma modules or dual Verma modules indexed by orbits. Since a maximal parabolic subgroup has minimal orbits, it is desirable to first compute the Casselman--Jacquet functor for maximal parabolics and then deduce the result for arbitrary parabolics by a transitivity argument for (dual) Verma modules. In Section~\ref{trans sec}, we carry out this transitivity and prove that the Casselman--Jacquet functor of a (dual) Verma module is again a (dual) Verma module; see Proposition~\ref{trans co-std} and Proposition~\ref{trans std} for details. 

\bigskip
\centerline{\scshape Acknowledgements}
We benefit a lot from the joint work~\cite{CWYZ} with Kei Yuen Chan and Hongfeng Zhang. Part of the manuscript was written during Wu's visit to CIRM for the thematic month on the Langlands program. Wu thanks the organizers, in particular Professor Raphaël Beuzart-Plessis, for their warm hospitality. Wu is partially supported by the New Cornerstone Science Foundation through the New Cornerstone Investigator Program awarded to Professor Xuhua He. Wu is also supported by the National Natural Science Foundation of China (Grant No. 123B1004). He also thanks Xuhua He for his continued support and encouragement.

\subsection{Notation and Conventions}
\subsubsection{General groups}
\begin{itemize}
    \item $G$, $H$, etc. (capital English letters): various real Lie groups (almost linear Nash groups or real reductive groups).
    \item $\mathfrak{g} := \mathrm{Lie}(G)_{\mathbb{C}}$ (Gothic letters): the complexified Lie algebras.
    \item $\delta_H$: the modular character of Lie group $H$.
    \item $Z_G$: the center of $G$.
    \item $\U(\mathfrak{g})$: the universal enveloping algebra of $\mathfrak{g}$;  $\mathcal{Z}(\mathfrak{g})$: the center of $\U(\mathfrak{g})$.
    
\end{itemize}

\subsubsection{Real reductive groups}
For a real reductive group $G$:
\begin{itemize}
    \item We fix a Cartan involution $\theta$ and a $\theta$-stable maximally split Cartan subgroup $A$ (from now on, the Cartan involution will no longer be involved and
 $\theta$ is free for other notation).
    \item $P^0 = L^0 N^0$: a minimal parabolic subgroup with Levi decomposition such that $L^0$ contains $A$.  
    \item $P = LN$: standard parabolic subgroup $P \supset P^0$ with Levi decomposition.
    \item $\overline{P}$: the opposite parabolic subgroup of $P$.
    \item $K$ (resp.\ $K_L$): the complexification of the maximal compact subgroup of $G$ (resp.\ $L$) fixed by the Cartan involution.
    \item $\mathfrak{b} \subset \mathfrak{p}^0$: a Borel subalgebra with $\mathfrak{a} \subset \mathfrak{b}$; $\fkn_B$: the nilradical of $\fkb$; $\fkn_L:= \fkn_B\cap \fkl$.
    \item $\Delta(\mathfrak{a}, \mathfrak{g})$: the roots of $\mathfrak{a}$ in $\mathfrak{g}$ with positive roots corresponding to $\mathfrak{b}$.
    \item $\rho$: the half-sum of positive roots.
    \item $\rho_{\mathfrak{l}}$: half-sum of positive roots in $\Delta(\mathfrak{a}, \mathfrak{l})$, where $L \subset P$ is a standard Levi subgroup.
    \item $W$ (resp.\ $W_L$): the Weyl group of $G$ (resp.\ $L$). For general linear groups, $W$ is represented by permutation matrices.
    \item $\chi_{\lambda}$: the infinitesimal character (algebra homomorphism $\mathcal{Z}(\mathfrak{g}) \to \mathbb{C}$) corresponding to $\lambda \in \mathfrak{a}^*$ via the Harish-Chandra isomorphism.% normalized so that $-\rho$ corresponds to the infinitesimal character of the trivial representation.
\end{itemize}

\subsubsection{General linear groups}
Let $\mathrm{GL}_n = \mathrm{GL}_n(\mathbb{K})$ where $\mathbb{K} = \mathbb{R}$ or $\mathbb{C}$:
\begin{itemize}
    \item Cartan involution: (complex-conjugate) transpose inverse; Cartan subgroup $A=A_n$: diagonal matrices.
    \item $B_n$: Borel subgroup of upper triangular matrices with unipotent radical $N_n$.
    \item $M_n$: mirabolic subgroup (matrices with last row $(0,\dots,0,1)$).
    \item $V_n$: unipotent radical of $M_n$ of matrices $\begin{pmatrix} I_{n-1} & v \\ & 1 \end{pmatrix}$.
    \item $H_{n,d}$: subgroup of $M_n$ of matrices $\begin{pmatrix} a & x \\ 0 & u \end{pmatrix}$ with $a \in \mathrm{GL}_{n-d}$, $u \in N_d$, and $x$ an $(n-d) \times d$ matrix.
    \item $P_{k,n-k}$: standard parabolic subgroup with Levi factor $\mathrm{GL}_k \times \mathrm{GL}_{n-k}$.
    \item $U_{k,n-k}$: unipotent radical of $P_{k,n-k}$.
\end{itemize}
For a subgroup $H \subseteq \mathrm{GL}_n$, $\overline{H}$ denotes the transpose of $H$. Fixed characters:
\begin{itemize}
    \item $\psi$: a fixed non-trivial unitary character of $\mathbb{K}$.
    \item $\psi_n$: character of $V_n$ defined by $\psi_n\left(\begin{bmatrix} I_{n-1} & v \\ & 1 \end{bmatrix}\right) := \psi(x_{n-1})$ for $v = [x_1,\dots,x_{n-1}]^t \in \mathbb{K}^{n-1}$; which also denotes the corresponding character of the Lie algebra $\mathfrak{v}_n$.
    \item $\psi_{n,d}$: character of $H_{n,d}$ defined by 
        \[
        \psi_{n,d}\left(\begin{bmatrix} a & x \\ 0 & u \end{bmatrix}\right) := \psi\left(\sum_{i=1}^{d-1} u_{i,i+1}\right), \quad u = (u_{ij})_{1 \leq i,j \leq d}.
        \]
\end{itemize}

\subsubsection{Topological vector spaces and representations}
The category of locally convex topological vector spaces that are complete and Hausdorff forms a quasi-abelian category, see~\cite[Section 2]{Ka93} for more discussions. However, the notion of exact sequences in this category differs from the standard categorical notion. Instead, we say that a sequence
\[
A \stackrel{\alpha}{\longrightarrow} B \stackrel{\beta}{\longrightarrow} C
\]
is exact if $\alpha$ and $\beta$ are continuous and the sequence is exact in the category of vector spaces.

Notations for the category of representations:
\begin{enumerate}
    \item When $G$ is an almost linear Nash group, we use $\Smod_G$ to denote the category of Fr\'echet representation of $G$, which is moderate-growth and smooth. We use $\widehat{G}$ to denote the irreducible classes in $\Smod_G$ that are unitarizable.
    \item When $G$ is a real reductive group, we use $\mathrm{CW}_G$ to denote the category of Casselman-Wallach representations of $G$.
    \item Let $(\fkg, H)$ be a Nash Lie pair. For $(\fkg, H)$-modules, two different situations arise. When $H$ is compact, a $(\fkg, H)$-module is a vector space equipped with compatible actions of $\fkg$ and $H$. The $H$-action may or may not be locally finite. When $H$ is non‑compact, a $(\fkg, H)$-module $\Pi$ is a Fréchet space equipped with compatible continuous actions of $\fkg$ and $H$, such that $\Pi \in \Smod_H$.
\end{enumerate}

Let $G$ be a Nash group and $E$ be a Hilbert representation of $G$. Then we use $E^{\infty}$ to denote the subspace consisting of smooth vectors. It is equipped with a countable family of semi-norms $q_W$ indexed by a basis of $W\in \U(\fkg)$:
\begin{equation}\label{smooth vector top}
    q_W(v):= \|W\cdot  v\|.
\end{equation}
Here, $\|\cdot\|$ is the norm of the Hilbert space. It is a standard fact that $E^{\infty}\in \Smod_G$. For a Harish-Chandra module $E$ of a real reductive group $G$, we will also use $E^{\infty}$ to denote the Casselman-Wallach globalization of $E$. This will not cause any confusion in the context.

Let $G$ be a Lie group and let $H$ be a closed subgroup. If $\sigma$ is a complete Hausdorff locally convex space equipped with a smooth action of $H$. Then we define the (normalized) smooth induction
\[
{}^{\infty}\Ind_{H}^G(\sigma):= \{ f\in \CC^{\infty}(G,\sigma)\mid f(hg)= \delta_G^{-\frac{1}{2}}(h) \delta_H^{\frac{1}{2}}(h) h\cdot f(g)\},
\]
equipped with the subspace topology of $\CC^{\infty}(G,\sigma)$. When $G$ is a Nash group, the smooth induction has subspaces $\CO\Ind_{H}^G(\sigma)\supset \SInd_H^G(\sigma)$ consisting of tempered sections and Schwartz sections, see~\cite[section 2.1]{Fd},\cite{CS21} for details. If $\sigma$ is a Hilbert representation of $H$, then we define the (normalized) $L^2$-induction
\begin{align*}
    \Ind_H^G:=\{& f\in \CM(G,\sigma)\mid f(hg)= \delta_G^{-\frac{1}{2}}(h) \delta_H^{\frac{1}{2}}(h) h\cdot f(g) \\
   & \int_{H\backslash G} \langle f(g), f(g)\rangle \eta(g) dg <\infty \},
\end{align*}
where $\CM(G,\sigma)$ is the space of measurable functions, and $dg$ is the right $\eta$-eigenmeasure. It is a Hilbert space equipped with the inner product defined by
\[
  \langle f_1, f_2 \rangle:= \int_{H\backslash G} \langle f_1(g), f_2(g)\rangle \eta(g) dg.
\]

\section{Preliminary}
\subsection{Filtration of a representation}
We introduce a generalized notion of filtration.
\begin{definition}\label{def_fil}
Given a representation $\sigma$ of an almost linear Nash group $G$, a \textbf{level $\leq 1$ filtration} of $\sigma$ consists of the data
\begin{itemize}
    \item[(i)] Finite decreasing subrepresentations of $\sigma$, \[\sigma= \sigma_0 \supset \sigma_1 \supset \dots \supset \sigma_m,\]
    \item[(ii)] For all $0 \leq i \leq m-1$, a finite or infinite decreasing chain of subrepresentations of $\sigma_i/\sigma_{i+1}$, 
    \[
    \sigma_i = \sigma_{i,0} \supset \sigma_{i,1} \supset \sigma_{i,2} \supset \dots \supset \sigma_{i+1},
    \]
    such that the canonical map
    $
    \sigma_i/\sigma_{i+1} \to \varprojlim_j \sigma_{i} / \sigma_{i,j}
    $
    is a topological isomorphism of $G$-representations.
\end{itemize}

A \textbf{level $\leq r$ filtration} of $\sigma$ consists of the data described above, with the additional requirement that each quotient $\sigma_{i,j} / \sigma_{i,j+1}$ is equipped with a level $\leq r-1$ filtration.

Given a level $\leq r$ filtration, for every pair of subrepresentations $\sigma^{\flat} \supset \sigma^{\sharp}$ in the filtration such that there are no other terms between $\sigma^{\flat}$ and $\sigma^{\sharp}$, we call the quotient $\sigma^{\flat} / \sigma^{\sharp}$ a \textbf{successive quotient} of the filtration.
\end{definition}

\subsection{Fr\'echet modules}
We first introduce the abstract notion of Fr\'echet algebra and Fr\'echet module following~\cite{Fd}. All Fréchet spaces are assumed to be nuclear, as all spaces in this paper are nuclear unless otherwise specified.
\begin{definition}
    \begin{enumerate}
    \item We call a pair $(\CA, \alpha: \CA \widehat{\otimes} \CA \to \CA)$ a \textbf{Fréchet algebra} if $\CA$ is a Fréchet space and $\alpha$ is a continuous map satisfying the algebra axioms. We will always omit $\alpha$ in the notation.
    \item Let $\CA$ be a Fr\'echet algebra. We call a pair $(\Pi, \alpha: \CA \widehat{\otimes} \Pi \to \Pi)$ a \textbf{Fréchet module} if $\Pi$ is a Fréchet space and $\alpha$ is a continuous map satisfying the module axioms. The image of $\alpha$ equipped with \textbf{quotient topology} is denoted by $\CA\cdot \Pi$.
    \item Let $(\Pi,\alpha)$ be a Fr\'echet module of $\CA$. We call it \textbf{differentiable} if $\alpha$ is surjective and $\{m\in \Pi\mid \CA\cdot m=0\}=0$. We call a Fr\'echet algebra differentiable if it is differentiable as a left module over itself.
    \item Let $(\Pi,\alpha)$ be a Fr\'echet module of $\CA$. We call it \textbf{factorizable} if $\alpha$ restricting to $\CA\otimes \Pi$ is surjective and $\{m\in \Pi\mid \CA\cdot m=0\}=0$. We call a Fr\'echet algebra \textbf{factorizable} if it is differentiable and all differentiable modules are factorizable.
    \item We call a Fr\'echet algebra \textbf{hereditary} if any closed submodule of its differentiable module is differentiable.
     \item Let $\CA$ be a Fréchet algebra. A sequence $\{e_n\}_{n=1}^{\infty}$ in $\CA$ is called an \textbf{approximation to the identity} if, for every $a \in \CA$, the sequences $\{e_n \cdot a\}$ and $\{a \cdot e_n\}$ converge to $a$.
    \end{enumerate}
\end{definition}

The following are our prototypes for Fréchet algebras and Fréchet modules.
\begin{example}
 Let $G$ be a Nash group with a fixed Haar measure $dg$. Then $\CS(G)$ is a Fr\'echet algebra under convolution product. Let $\{f_n\}_{n=1}^{\infty}$ be a sequence in $\CS(G)$ such that 
        $$\lim_{n\to \infty}\supp(f_n)=\{e\}\text{ and }\int_{G}f_n(g)dg=1.$$
            Then it is an approximation to identity. In addition, it is a standard fact that there is a category equivalence between $\Smod_G$ and the category of differentiable $\CS(G)$-module. Hence, by Diximier-Malliavin theorem, $\CS(G)$ is a hereditary and factorizable Fr\'echet algebra.
\end{example}

\begin{example}
    Let $X$ be a Nash manifold. Then $\CS(X)$ is a Fr\'echet algebra under point-wise multiplication. Let $\{K_n\}_{n=1}^{\infty}$ be an increasing exhaustion of $X$ by open relatively compact subsets. Then $\{f_n\}_{n=1}^{\infty}\subset \CC_c^{\infty}(X)$, such that 
    \[
    \supp(f_n)\subset K_{n+1}  \text{ and } f_n(x)=1 \text{ for } x\in K_n,
    \]
     is an approximation to identity. We can find a finite Nash open covering $\{U_i\}$ such that each $U_i$ is Nash isomorphic to a closed Nash sub-manifold of $\BR^d$ for some $d$. Then each $\CS(U_i)$ is quotient algebra of $\CS(\BR^d)$, which is hereditary and factorizable by Fourier transform. Consequently, $\CS(X)$ is also hereditary and factorizable by partition of unity.
\end{example}
\begin{example}\label{Sobolev example}
    We demonstrate a Fr\'echet module of $\CS(\BR)$ that will be used in the counter-example section~\ref{counter-exa-sec}. Let $x$ be the coordinate of $\BR$, and let $\xi$ be the coordinate of $\BR^*$. Let $k$ be a positive integer and $H_k(\BR)$ be the subspace of tempered distributions $D\in\CS(\BR)'$ such that
    \[
     \left(\frac{\partial}{\partial x}\right)^i D\in L^2(\BR), \text{ for } 0\leq i\leq k.
    \]
    Since as a tempered distribution, the Fourier transform of $f\in L^2(\BR)$ still belongs to $L^2(\BR)$(\cite[Theorem 9.2.2]{FJ98}), we can define an inner product on $H_k(\BR)$ by
    \[
    \langle D_1, D_2\rangle:= \int_{\BR^*} (1+|\xi|^2)^{k}  \widehat{D_1} \ov{\widehat{D_2}} \, d\xi.
    \]
    Equipped with this inner product, $H_k(\BR)$ is a Hilbert space, which is usually called the order $k$ Sobolev space. Note that for $p>q$, there is a continuous embedding $H_p(\BR)\hookrightarrow H_q(\BR)$. Let 
    \[
    H_{\infty}(\BR): = \varprojlim_k H_k(\BR)= \bigcap_k H_k(\BR) ,
    \]
    which is equipped with the inverse limit topology. Then $H_{\infty}(\BR)$ is a Fr\'echet space.  It is obvious that $\CS(\BR)\subset H_{\infty}(\BR)$ and $H_{\infty}(\BR)$ is a Fr\'echet $\CS(\BR)$-module under point-wise multiplication. Moreover, by~\cite[Corollary 9.3.4]{FJ98}, any function in $H_{\infty}(\BR)$ is smooth and $\CS(\BR)\cdot H_{\infty}(\BR)= \CS(\BR)$. Therefore, it is not differentiable. Actually, it is also not nuclear, see Lemma~\ref{non-nuclear lem}.
\end{example}

\begin{remark}
    Let $N$ be an abelian unipotent Nash group. For any $\beta\in\Smod_N$, we regard it as a differentiable $\CS(N)$-module. By the Fourier transform, we also regard it as a differentiable $\CS(\widehat{N})$-module.
\end{remark}

We review some results in~\cite{WZ} splitting two different kinds of Fr\'echet modules in a filtration, which is essential for the rearrangement of the coarse spectral filtration.

\begin{lemma}[\cite{WZ}, Lemma 2.40]\label{limit diff}
    Let $\CA$ be a hereditary and factorizable Fr\'echet algebra. Let $\{\Pi_k,d_k: \Pi_{k+1}\to \Pi_k\}_{k=0}^{\infty}$ be an inverse system of Fr\'echet $\CA$-module such that $\Pi_0=0$ and $d_k$ is surjective for any integer $k$.
    \begin{enumerate}
        \item Suppose that each $\ker d_k$ is a differentiable $\CA$-module, then $\varprojlim_k \Pi_k$ is a differentiable $\CA$-module as well.
        \item Suppose that $\CA\cdot \Ker d_k=0$, then
        \[
        \CA\cdot \varprojlim_k \Pi_k=0
        \]
    \end{enumerate}
\end{lemma}

\begin{lemma}[\cite{WZ}, Lemma 2.41]\label{split}
    Let $\CA$ be a differentiable Fr\'echet algebra possessing an approximation to the identity. Let 
   \[
    0\lra \Pi_1\lra \Pi\stackrel{\varphi}{\lra} \Pi_2\lra 0
    \]
     be a short exact sequence of Fr\'echet $\CA$-module. Suppose that $\CA\cdot \Pi_1=0$ and $\Pi_2$ is differentiable. Then there is a splitting of $\varphi$. 
\end{lemma}

The following lemma is crucial in proving the exactness of derivative functors.
\begin{lemma}\label{exa_lem}
    Let $\CA$ be a Fr\'echet algebra possessing an approximation to identity. Considering an exact sequence of Fr\'echet $\CA$-modules
    \[
    0\lra \pi_1\lra \pi_2\stackrel{\varphi}{\lra} \pi_3\lra 0,
    \]
    if $\CA\cdot \pi_1$ is closed in $\pi_1$, then we have a short exact sequence
    \[
    0\lra \CA\cdot \pi_1\lra \CA\cdot \pi_2\lra \CA\cdot \pi_3\lra 0.
    \]
\end{lemma}
\begin{proof}
    It suffices to prove the exactness at $\CA\cdot \pi_2$. Let $v\in \CA\cdot \pi_2$ such that $\varphi(v)=0$, then $v\in \pi_1$. Equipped $\CA\cdot \pi_2$ with quotient topology, then it is a differentiable Fr\'echet module of $\CA$. Let $\{e_n\}_{n=1}^{\infty}$ be an approximation to identity of $\CA$. By \cite[Lemma 2.3.6]{Fd}, we have $\lim_{n\to \infty} e_n\cdot v=v$ in $\CA\cdot \pi_2$, hence in $\pi_2$. Note that $\CA\cdot \pi_1$ is closed in $\pi_1$, hence closed in $\pi_2$, and $e_n\cdot v\in\CA\cdot \pi_1$. Consequently, we have $v\in \CA\cdot \pi_1$. 
\end{proof}

A key ingredient in~\cite{Fd} to describe the irreducible representation of almost linear Nash groups is the (smooth) imprimitive system. Let $X$ be a Nash manifold with a Nash action by an almost linear Nash group $G$. Then $(E,\pi)\in \Smod_G$ is called an imprimitive system based on $X$ if $E$ is equipped with a $\CS(X)$-action $\alpha$ satisfying
\[
\alpha (g\cdot f) \cdot (\pi(g)\cdot v)=  \pi(g) \cdot \alpha(f)\cdot v
\]
for any $v\in E, g\in G$ and $f\in \CS(X)$. The category of imprimitive systems is denoted by $\Rep_{G,X}^{\infty}$, whose morphisms are continuous linear maps intertwining $G$ and $\CS(X)$-action.
\begin{example}\label{imprim exa}
    Let $H$ be a subgroup of $G$, and let $\sigma\in\Smod_{H}$. Then $\SInd_{H}^G(\sigma)$ is an imprimitive system based on $H\backslash G$. Here, $\CS(H\backslash G)$ acts on $\SInd_{H}^G(\sigma)$ by point-wise multiplication.
\end{example}

\subsection{Category $\CC(\fkg,L)_f$.}
In this subsection, we recall the smooth category $\CC(\fkg,L)_f$ defined in~\cite{WZ} and introduce a class of co-standard objects of it. Readers are referred to~\cite[section 2.5]{WZ} and section~\ref{alg complete cat sec} for some basic notations and definitions.

Let $Q\supset P$ be a pair of standard parabolic subgroups with standard Levi decompositions $P=LN$ and $Q=MU$. Let $V:=M\cap N$.  Let $F_{k}$ be the increasing filtration on $\U(\fku)$ defined as the image of $\oplus_{i=0}^k T^i(\fku)$. Then the well-known PBW theorem establishes the topological isomorphism as $M$-representations
\[
   \U(\fku) \stackrel{\simeq}{\lra} \mathrm{gr}_{F} \U(\fku)\simeq \Sym(\fku),
\]
where we equip limit Fr\'echet topology on two sides. Hence, $\Sym(\fku)'\simeq \BC[[\fku]]$ as $M$-representations.

\begin{definition}
    Let $\tau$ be a Fr\'echet representation of $\fkm$. The \textbf{dual Verma module} of $\tau$ is defined as
    \[
    \CQ(\tau) := (\U(\fkg)\otimes_{\U(\fkq)} \tau')'.
    \]
    Here, $\U(\fkg)$ is equipped with the limit Fr\'echet topology, and $\U(\fkg)\otimes_{\U(\fkq)} \tau'$ is a quotient of the projective tensor product topology, which is complete. 
\end{definition}

 It is obvious that $\CQ(\cdot)$ is an exact covariant functor.
\begin{proposition}\label{dual verma prop}
    Let $\tau\in\CC(\fkm,L)_f$. Then $\CQ(\tau)\in \CC(\fkg,L)_f$. 
\end{proposition}
   \begin{proof}
       By the classification of irreducible objects, $\tau$ is $\CZ(\fkm)$-finite since $\tau \in \CC(\fkm,L)_f$. Hence $\CQ(\tau)$ is $\CZ(\fkg)$-finite. To prove the proposition, it suffices, by~\cite[Lemma 2.24]{WZ}, to show that $\CQ(\tau) \in \CC(\fkg,L)$. 

       As representations of $L$, we have
       \begin{align*}
        &  \CQ(\tau)\simeq (\U(\ov{\fku})\otimes \tau')'\simeq \left(\varprojlim_k ( \Sym(\fku)/\fku^k\Sym(\fku))\right)\widehat{\otimes} \tau   \\
         & \simeq \left(\prod_{\alpha\in \fkz_L^*} \Sym(\fku)_{\alpha} \right) \widehat{\otimes} \left( \prod_{\beta\in \fkz_L^*} \tau_{\beta} \right)\simeq \prod_{\gamma\in\fkz_L^*}\bigoplus_{\alpha+\beta =\gamma} \Sym(\fku)_{\alpha}\otimes \tau_{\beta},
       \end{align*}
      since $\Sym(\fku)_{\alpha}$ is a finite dimensional space. Note that there exists a finite set $\mathscr{S}\subset \fkz_L^*$ such that $\wt(\tau)\subset \mathscr{S}+\Omega_{\fkv}$. Thus, 
      \[
      \wt( \CQ(\tau) )\subset  \mathscr{S}+ \Omega_{\fkn}.
      \]
     Consequently, $\CQ(\tau)^{[\fkz_L]}$ is $\ov{\fkn}$-finite and for each $\gamma\in \fkz_L^*$, $\CQ(\tau)_{\gamma}$ is a Casselman-Wallach representation of $L$.
   \end{proof}

\begin{remark}
     When $\tau$ is not irreducible, $\CQ(\tau)$ is generally different from the standard dual of $\CV(\tau)$. The standard dual (\cite[(47)]{HS83b}), on the other hand, corresponds under localization to the Verdier dual. This is because $\CQ(\cdot)$ is covariant, whereas the standard dual should be contravariant. However, we will see in future work that the dual Verma module is precisely the object that appears in the computation of the Casselman-Jacquet functor for parabolically induced representations, see~\cite{WZ26}.
\end{remark}

Assume $P=Q$ from now on. In order to prove Lemma~\ref{unique sub}, we record a topological lemma.
\begin{lemma}\label{duality lem}
    Let $E$ be a nuclear Fr\'echet space. Suppose $W_1,W_2$ are closed subspaces of $V$, then in $E'$,
    \[
    (W_1 \cap W_2)^{\perp}= \ov{W_1^{\perp}+W_2^{\perp}}.
    \]
\end{lemma}
\begin{proof}
   Since the (dual) nuclear Fr\'echet space is reflexive, $(W^{\perp})^{\perp}=W$ for any closed subspace of $E$ or $E'$. Hence, the equality follows from the fact that
   \[
  ( \ov{W_1^{\perp}+W_2^{\perp}})^{\perp}= W_1\cap W_2.
   \]
\end{proof}

\begin{lemma}\label{unique sub}
    Let $\tau$ be an irreducible representation of $L$. Then $\CQ(\tau)$ has a unique non-zero irreducible closed submodule.
\end{lemma}
\begin{proof}
    It suffices to show that for two non-zero closed submodules $W_1,W_2$, we have $W_1\cap W_2\neq 0 $. By Lemma~\ref{duality lem}, it suffices to show 
$\ov{W_1^{\perp}+W_2^{\perp}} \neq \CQ(\tau)'$. By the reflexive property of nuclear space, $\tau'$ is topologically irreducible as an $L$-representation. Note that $\CQ(\tau)'\simeq \U(\fkg)\otimes_{\U(\fkp)}\tau'$ is $\fkz_L$-finite and generated by $\tau'$ as a $\U(\fkg)$-module. Consequently, this implies $( \ov{W_1^{\perp}+W_2^{\perp}} )\bigcap \tau'=0$.
\end{proof}

Similar to the proof of Proposition~\ref{dual verma prop}, it is not hard to see this irreducible submodule lies in the category $\CC(\fkg,L)_f$. On the other hand, by the proof of Lemma~\ref{unique sub}, its minimal $\fkz_L^*$-weight space is $\tau$. In conclusion, we reach the following theorem.
\begin{theorem}\label{unique sub thm}
   $\CL(\tau)$ is the unique irreducible subobject of $\CQ(\tau)$.
\end{theorem}

\subsection{Computation of twisted Jacquet functors}
We state a well-known lemma about the group co-invariant and Lie algebra co-invariant of a smooth moderate-growth Fr\'echet representation $\sigma$ of a  unipotent Nash group $N$. Define
\[
N\cdot \sigma:= \Span\{ \sigma(n) v-v\mid v\in\sigma,n\in N\}
\]
and $\fkn \cdot \sigma:=\Span\{ \sigma(X) \cdot v\mid X\in \fkn, v\in \sigma\}$.
\begin{lemma}[Theorem 7.7, \cite{CS21}]\label{coinv equality}
    Let $\sigma\in \Smod_N$. Then $N\cdot \sigma=\fkn \cdot \sigma$.
\end{lemma}
The computation of the twisted Jacquet functor in the literature mainly reduces to the case where the Lie algebra acts by a family of characters. In the next proposition, we extend this to the case where the Lie algebra action is twisted by a finite-dimensional nilpotent representation.
\begin{proposition}\label{twisted Jacquet computation}
       Let $M$ be a Nash manifold. Let $\fkv$ be a real abelian Lie algebra such that $\dim_{\BR} \fkv =\dim M$, and let $(\gamma,W)$ be a finite-dimensional nilpotent representation of $\fkv$. Given a Nash morphism
    \[
    \epsilon: M \lra \fkv^* 
    \]
    such that either $0\notin \epsilon(M)$ or $0 \in \fkv^*$ is a regular value of $\epsilon$. Let $X\in\fkv$ act on $f\in \CS(M,W)$ by 
    \begin{equation}\label{action eq}
         ( X\cdot f)(m):= \left(\epsilon(m)(X)  +\gamma(X)\right) f(m).  
    \end{equation}
    Then
    \begin{enumerate}
        \item $ \rmh_i(\fkv, \CS(M,W))=0$ when $i>0$, and
        \item there is a natural isomorphism $\rmh_0(\fkv, \CS(M,W))\simeq \CS(\epsilon^{-1}(0),W)$.
        \end{enumerate} 
\end{proposition}
\begin{proof}
    Consider a decreasing filtration of $W$ such that $\fkv$ acts on each successive quotient trivially. Then, the statement (1) follows directly from~\cite[Lemma 6.2.2]{AGS15b}. 

    The key point is to define the natural isomorphism in the statement (2). We define, locally near each point of $\epsilon^{-1}(0)$, a linear automorphism $\varphi$ on $\CS(M,W)$, such that after applying $\varphi$, $X\in\fkv$ acts on $f\in\CS(M,W)$ by
    \[
   ( X\star f)(m) := \epsilon(m)(X) f(m).
    \]
    Then by~\cite[Lemma 6.2.2]{AGS15b} the natural map is induced by 
    \[
    \CS(M,W) \stackrel{\varphi}{\lra} \CS(M,W)\stackrel{\mathrm{Res_{\epsilon^{-1}(0)}}}{\lra} \CS(\epsilon^{-1}(0),W).
    \]
    
    Now, we define $\varphi$. By assumption, $\epsilon^{-1}(0)$ is a discrete subset of $M$ and for any $p\in \epsilon^{-1}(0)$, there exists an open neighborhood $\mathscr{U}_p$ such that $\epsilon$ is invertible on $\mathscr{U}_p$. Let 
    \[
    \mathscr{U}:= \bigcup_{p\in \epsilon^{-1}(0)} \mathscr{U}_p \text{ and } \mathscr{V}:= M\setminus \epsilon^{-1}(0).
    \]
    Then by a similar covering argument for $\mathscr{U}$ and $\mathscr{V}$ as in~\cite[Corollary 4.3]{WZ}, we have
    \[
    \rmh_0(\fkv,\CS(M,W))\simeq \rmh_0(\fkv, \CS(\mathscr{U},W))= \prod_{p\in \epsilon^{-1}(0)} \rmh_0(\fkv,\CS(\mathscr{U}_p,W)).
    \]
    We identify $\fkv^*$ with the translation invariant vector field on $\fkv^*$. For a fixed $p\in \epsilon^{-1}(0)$, we have a linear map by push-forward of vector fields:
    \[
    \epsilon^{-1}_*: \fkv^* \lra \CN( \mathscr{U}_p, T\mathscr{U}_p),
    \]
    where $\CN( \mathscr{U}_p, T\mathscr{U}_p)$ consists of Nash sections of the tangent bundle. On the other hand, 
    \[
  \theta:  \Hom_{\BC}(\fkv^*, \CN( \mathscr{U}_p, T\mathscr{U}_p)) \simeq \fkv\otimes \CN( \mathscr{U}_p, T\mathscr{U}_p) \lra \End(\CS(\mathscr{U}_p,W)),
    \]
    where the second map is given by $(X\otimes D)(f):= \gamma(X) D(f)$ for $X\in \fkv$, $D\in \CN( \mathscr{U}_p, T\mathscr{U}_p)$ and $f\in \CS(\mathscr{U}_p,W)$. Hence,
    \[
     \varphi:= \exp(- \theta(\epsilon^{-1}_*) )\in \GL (\CS(\mathscr{U}_p,W))
    \]
    is what we want. We check that
    \begin{equation}\label{intertwin eq}
             X\star  \varphi (f) = \varphi (X\cdot f).
    \end{equation}
    Fix a basis $X_i$ of $\fkv$. The dual basis in $\fkv^*$ is denoted by $e_i$. Then,
    \[
    \theta(\epsilon^{-1}_*)=\sum_i X_i \otimes \epsilon^{-1}_*(e_i).
    \]
    Therefore, we only need to check equation~\eqref{intertwin eq} for $X=X_i$ and it basically follows from the equation
    \[
    \frac{1}{k} \theta(\epsilon^{-1}_*)^k(\epsilon(m)(X_i) f) =  \frac{1}{k} \epsilon(m)(X_i)\theta(\epsilon^{-1}_*)^k( f) + \theta(\epsilon^{-1}_*)^{k-1}\gamma(X_i)( f).
    \]
\end{proof}

\section{Canonical Bernstein-Zelevinsky filtration}
\subsection{General theory}\label{general theory section}
Let $G$ be a real reductive group, and let $P$ be a parabolic subgroup of $G$ such that 
\begin{itemize}
    \item The unipotent radical $N$ is abelian, and
    \item $P$ has finitely many orbits on $\widehat{N}$.
\end{itemize}
In this section, we establish the canonical Bernstein-Zelevinsky filtration for a Casselman-Wallach representation $\pi$ restricting to $P$ under the following assumption:
\begin{itemize}
    \item[$\spadesuit$]\label{item:A1} any principal series of $G$ restricted to $P$ has a \textbf{coarse spectral filtration}. 
\end{itemize}
We first introduce the \textbf{coarse spectral filtration}. Take a stratification of $\widehat{N}$:
\[
Z_{0}=\emptyset \subset Z_1=\{0\}\subset \dots \subset Z_{r}=\widehat{N}
\]
such that $Z_\ell$ is a closed subset of $\widehat{N}$ and $\CO_{\ell}:=Z_{\ell+1}\setminus Z_{\ell}$ is an open $L$-orbit in $Z_{\ell+1}$ for every $0\leq \ell\leq r-1$. We fix a set of characters $\phi_{\ell}$ such that $\phi_{\ell}\in \CO_{\ell}$. The stabilizer of $\phi_{\ell}$ is simply denoted by $S_{\ell}$. 
\begin{definition}
   Let $\sigma\in \Smod_{S_{\ell}\cap L}$. By trivial extension, it is also regarded as a representation of $S_{\ell}$. The \textbf{Mackey induction} of $\sigma$ with respect to $\phi_{\ell}$ is 
    \[
    I_{\ell}(\sigma):= \SInd_{S_{\ell}}^P (\sigma\boxtimes  \phi_{\ell}).
    \]
\end{definition}

\begin{definition}[Coarse spectral filtration]\label{coarse def}
     Let $\pi\in\Smod_P$. A coarse spectral filtration is a filtration as in Definition~\ref{def_fil}, where each successive quotient is a \textbf{Mackey induction}.
\end{definition}

Our main result of this subsection is the following.
\begin{theorem}\label{Canonical filtration}
    Let $\pi$ be a Casselman-Wallach representation of $G$. Under the assumption $\spadesuit$, $\CS(\widehat{N}\setminus Z_{\ell})\cdot \pi$ is closed in $\pi$ for every $1\leq \ell\leq r$. In other words, $\pi$ has a decreasing filtration as $P$-representations
    \begin{equation}\label{con-fil-eq}
          \pi=\pi_0\supset \pi_1\supset \dots \supset \pi_r= 0 
    \end{equation}
    such that $\pi_{\ell}\simeq \CS(\widehat{N}\setminus Z_{\ell})\cdot \pi$. Moreover, we have functorial isomorphisms
    \begin{enumerate}
        \item $\pi_0/\pi_1\simeq \mathrm{CJ}^{\infty}_N(\pi)$;
        \item and $\pi_{\ell}/\pi_{\ell+1}\simeq \SInd_{S_{\phi}}^P(\Psi_{\phi}(\pi))$ if $\CO_{\ell}$ is an open subset of $\widehat{N}$, where $\phi\in \CO_{\ell}$.
    \end{enumerate}
\end{theorem}

We first prove the theorem for principal series, and then generalize to Casselman-Wallach representations through Casselman's embedding theorem. For the preparation of the proof, we demonstrate some examples of differentiable module of Fr\'echet algebra $\CS(\widehat{N}\setminus Z_{t})$ for a fixed integer $t$.
\begin{example}\label{diff-mod-exam}
   We fix a set of characters $\phi_{\ell}$ such that $\phi_{\ell}\in \CO_{\ell}$.  Then 
   \begin{itemize}
       \item $I_{\ell}(\sigma)$ is a differentiable $\CS(\widehat{N}\setminus Z_{t})$-module if $\ell\geq t$, and
       \item $I_{\ell}(\sigma)$ is killed by $\CS(\widehat{N}\setminus Z_{t})$ if $\ell<t$.
   \end{itemize}
   We realize $I_{\ell}(\sigma)$ as a tempered bundle over $S_{\ell}\backslash P\simeq \CO_{\ell}$. Then by Fourier inversion theorem, we will see that the action of  $\CS(\widehat{N}\setminus Z_{t})$ is given by 
    \[
   ( f\cdot \xi)(x)=f(x)\xi(x), \text{ for } f\in\CS(\widehat{N}\setminus Z_t), \xi\in I_{\ell}(\sigma) \text{ and } x\in \CO_{\ell}.
    \]
    Since $\CS(\widehat{N}\setminus Z_{\ell})$ is contained in $\CS(\widehat{N}\setminus Z_{t})$ when $\ell\geq t$, the result follows from the fact that the restriction map
    \[
    \CS(\widehat{N}\setminus Z_{\ell})\lra \CS(\CO_{\ell})
    \]
    is surjective. 
\end{example}

\subsubsection{Proof of Theorem~\ref{Canonical filtration} for generalized principal series} \label{generalized_prin-subsec}

In this subsection, we assume that $\pi$ is a generalized principal series, i.e., a representation induced from a finite-dimensional representation of $P^0$. We first establish the existence of the filtration in~\eqref{con-fil-eq} by a rearranging statement stronger than~\cite[Proposition 7.11]{WZ}.
\begin{lemma}\label{generalized prin lem}
    Let $\pi$ be a generalized principal series representation of $G$. Under the assumption $\spadesuit$, the restriction $\pi|_P$ admits a filtration of $P$-representations
    \[
    \pi|_{P} = \sigma_0 \supset \sigma_1 \supset \dots \supset \sigma_r = 0
    \]
   such that the successive quotients in $\sigma_{\ell}/\sigma_{\ell+1}$ are of the form $I_{\ell}$.
\end{lemma}
\begin{proof}
   Since every generalized principal series admits a filtration whose subquotients are principal series, it follows that any generalized principal series also possesses a coarse spectral filtration. Let $\beta$ be a subquotient of coarse spectral filtration of $\pi|_P$ with a short exact sequence
    \begin{equation}\label{reversed eq}
            0\lra \beta_1 \lra \beta\lra \beta_2\lra 0
    \end{equation}
    in coarse spectral filtration of $\pi|_P$. Suppose that the successive quotients in $\beta_1$ are of the form $I_0$ and the successive quotients in $\beta_2$ are of the form $I_{\ell},\ell\neq 0$. Then by Example~\ref{diff-mod-exam} and Lemma~\ref{limit diff}, we have $\beta_1$ is killed by $\CS(\widehat{N}\setminus \{0\})$ and $\beta_2$ is a differentiable $\CS(\widehat{N}\setminus \{0\})$-module. Therefore, by Lemma~\ref{split}, we have a reversed short exact sequence as $P$-representations
    \[
    0\lra \beta_2 \lra \beta\lra \beta_1\lra 0.
    \]
    Applying this argument successively, we can get $\sigma_1$. Inductively, we can obtain the desired subrepresentation $\sigma_{\ell+1}$ through the $\CS(\widehat{N}\setminus Z_{\ell+1})$-action on the filtration of $\sigma_{\ell}$.
\end{proof}
The filtration~\eqref{con-fil-eq} follows from Lemma~\ref{generalized prin lem} directly, since $\CS(\widehat{N}\setminus Z_{\ell})\cdot\pi=\sigma_{\ell}$. The statement (1) follows from the following theorem, which is one of the main results in~\cite{CWYZ}.
\begin{theorem}[Corollary 6.2, Proposition 5.1,~\cite{CWYZ}]\label{topological spectral short exa}
    Let $\pi\in\Smod_{N}$. Then there is a short exact sequence
    \[
    0 \lra \ov{\CS(\widehat{N}\setminus\{0\})\cdot \pi} \lra \pi \lra \mathrm{CJ}_N^{\infty}(\pi) \lra 0.
    \]
\end{theorem}

We then prove statement (2) of Theorem~\ref{Canonical filtration}. 
\begin{lemma}\label{open-orbits split}
    Let $s$ be an integer such that $\CO_{\ell}$ is an open orbit in $\widehat{N}$ for $\ell\geq s$. Let $\Omega_s:=\bigcup_{\ell\geq s} \CO_{\ell}$, then
    \[
    \CS(\Omega_s)\cdot \pi \simeq \bigoplus_{\ell\geq s} \CS(\CO_{\ell})\cdot \pi.
    \]
\end{lemma}
\begin{proof}
 On the one hand, we have a continuous surjective map given by addition
\[
\varphi:\bigoplus_{\ell\geq s} \CS(\CO_{\ell}) \cdot \pi\lra \CS(\Omega_s)\cdot \pi.
\]
Suppose that $(v_{\ell})\in \Ker\varphi$ such that there exists $i\geq s$ such that $v_i\neq 0$. Then in $\CS(\Omega_s)\cdot \pi$, $v_i = \sum_{\ell\neq i} v_{\ell}$. Let $\{e_n\}$ be an approximation to the identity of $\CS(\CO_i)$. Then, $\{e_n\cdot v_i\}$ converges to $v$ by~\cite[Lemma 2.3.6]{Fd}. However, $e_n\cdot v_{\ell}=0$ for any $n$ and $\ell\neq i$. This leads to a contradiction. 

Consequently, $\varphi$ is injective and hence an isomorphism by the open mapping theorem.    
\end{proof}

Let $\CO_{\ell}$ be an open orbit, and let $\phi\in \CO_{\ell}$. Lemma~\ref{open-orbits split} demonstrates that
\[
 \pi_{\ell}/\pi_{\ell+1}\simeq \CS(\CO_{\ell}) \cdot \pi.
\]

In what follows, we show that the twisted Jacquet functor of open orbits is always Hausdorff (without the assumption of $\spadesuit$) and we give a description of this spectrum using the twisted Jacquet functor and Schwartz induction. We remark that any unipotent Nash group is connected~\cite[Theorem 1.2(1)]{Sun15}.
\begin{proposition}\label{open-iso}
      Let $P$ be an almost linear Nash group such that $N$ is abelian. Let $\pi\in\Smod_P$. Suppose that $\CO$ is an open $P$-orbit in $\widehat{N}$. Then $\Psi_{\phi}(\pi)$ is Hausdorff and there is a functorial isomorphism
    \[
   \Xi: \CS(\CO)\cdot \pi\simeq \SInd_{S_{\phi}}^P(\Psi_{\phi}(\pi)).
    \]
\end{proposition}
\begin{proof}
 \textbf{Step 1: Prove that the twisted Jacquet functor is Hausdorff by the action map $a: \CS(\CO)\widehat{\otimes}\pi \lra \pi$.}

When $\phi$ is clear from the context, we will omit the subscript $\phi$ for simplicity. Consider the following composition of continuous $P$-equivariant maps
\begin{equation}\label{identification}
    \SInd_{S}^{P}(\pi)\stackrel{\alpha}{\lra } \SInd_{S}^{P}(\BC)\widehat{\otimes}\pi \stackrel{\beta}{\lra } \CS(\CO)\widehat{\otimes}\pi,
\end{equation}
where $\alpha$ is the isomorphism given by
\[
f\mapsto (p\mapsto p^{-1}f(p))\text{ for }p\in P,f\in \SInd_{S}^{P}(\pi)
\]
and $\beta$ is the isomorphism given by the Nash isomorphism $\mathrm{Is}:S_{\phi}\backslash P\simeq  \CO_{\phi}$. The pull-back of the measure on $\CO_{\phi}$ is denoted by $dp$. Since 
$$\ov{\SInd_{S}^{P}(N(\pi\otimes \phi^{-1}))}\subset \SInd_{S}^{P}(\ov{N(\pi\otimes \phi^{-1})}), $$ 
it has a dense subspace 
\[
\Span_{\BC}\left\{f_w:g\mapsto \int_{S} f(hg)h^{-1}\cdot w dh\mid f\in\CS(P),w\in  \fkn(\pi\otimes \phi)\right\}.
\]
\begin{itemize}
    \item Claim 1: Let $\widetilde{a}:=  a\circ\beta\circ \alpha$. Then $\SInd_{S}^{P}(N(\pi\otimes \phi^{-1}))\subset \ker(\widetilde{a})$.
\end{itemize}
The claim follows from the detailed computation:
\begin{equation*}
    \begin{split}
   a\circ\beta&\circ \alpha(f_w)=\int_{N} \int_{S\backslash P} n p^{-1}f_w(p)(p\cdot \phi)(n)^{-1}dp dn\\
   &=\int_{N}\int_{S\backslash P}\int_{S} n(hp)^{-1} \cdot wf(hp)(p\cdot \phi)(n)^{-1} dh dp dn.
    \end{split}
\end{equation*}
Let $w=\sum_{j=1}^J \phi(n_j)^{-1}n_jt_j-t_j$, where $n_j\in N$ and $t_j\in\pi$. Since the integration is linear on $w$, $\widetilde{a}(f_w)=0$ follows from
\begin{equation*}
    \begin{split}
       & \int_{N}\int_{S\backslash P}\int_{S} n(hp)^{-1} n_j\cdot t_j f(hp)(p\cdot \phi)(n)^{-1}\phi(n_j)^{-1} dh  dp  dn\\
       = \int_{N}\int_{S\backslash P}&\int_{S}  n \Ad_{(hp)^{-1}}(n_j) (hp)^{-1}t_j f(hp)(p\cdot\phi)( n \Ad_{(hp)^{-1}}(n_j))^{-1}dh  dp  dn\\
       =\int_{N}&\int_{S\backslash P}\int_{S} n(hp)^{-1} \cdot t_jf(hp)(p\cdot \phi)(n)^{-1} dh dp dn,
    \end{split}
\end{equation*}
where the third equality follows from the change of variable $n\mapsto n\Ad_{(hp)^{-1}}(n_j^{-1})$. 
\\

Note that $\widetilde{a}$ is a morphism in the category of imprimitive system $\Rep^{\infty}_{P,S\backslash P}$, see Example~\ref{imprim exa}.  Since $\SInd_{S}^{P}(\pi)$ is a differentiable $\CS(S\backslash P)$-module and $\CS(S\backslash P)$ is a hereditary algebra, by~\cite[Theorem 2.5.8]{Fd}, there is a closed $S$-subrepresentation $\kappa$ of $\pi$ such that
\[
 \ker(\widetilde{a})= \SInd_{S}^{P}(\kappa). 
\]
\begin{itemize}
    \item Claim 2: $\kappa= N(\pi\otimes \phi^{-1})$.
\end{itemize}
\begin{proof}[Proof of the claim]
    The action map $\widetilde{a}$ induces $\SInd_{S}^P(\pi/\kappa)\to \pi$. Taking the twisted Jacquet functor (with respect to group $N$) on the two sides, and by \textbf{Claim 1} and the proof of~\cite[Corollary 7.6]{WZ}, we obtain
    \[
    \pi/\kappa \xleftarrow[\simeq]{ev_{e}} \Psi(\SInd_{S}^P(\pi/\kappa))\xrightarrow{\Psi(\widetilde{a})}  \Psi(\pi),
    \]
    where $ev_{e}$ is evaluation at the identity point. Hence, to prove Claim 2, it suffices to show that $\Psi(\widetilde{a}) \circ ev_{e}^{-1}$ is induced by the identity map on $\pi$. Let $f\in \SInd_{S}^P(\pi)$ such that $f(e)=v\in \pi$. Then in $\Psi(\pi)$, we have
     \begin{equation*}
        \begin{split}
       \widetilde{a}&(f)=\int_{N}\int_{S\backslash P}  n\alpha(f)(p)(p\cdot\phi)(n)^{-1}dpdn=\int_{N}\int_{S\backslash P}\phi(n)\alpha(f)(p)(p\cdot\phi)(n)^{-1} dpdn\\
       &=\int_{N}\int_{\widehat{N}}(\alpha\circ\beta)(f)(\xi) (-\xi+\phi)(n)d\xi dn =(\alpha\circ\beta)(f)(\phi)=f(e)=v,
        \end{split}
    \end{equation*}
    where the last line follows from the Fourier inversion theorem.
\end{proof}

\textbf{Step 2: Prove the functorial isomorphism $\Xi:\CS(\CO)\cdot \pi\simeq \SInd_{S_{\phi}}^P(\Psi_{\phi}(\pi))$.}

Lemma~\ref{coinv equality} and \textbf{Claim 2} ensure that $\ker(\widetilde{a})= \SInd_{S}^P(\fkn(\pi\otimes \phi))$. Hence, by the exactness of the Schwartz induction, the desired isomorphism is induced by $\widetilde{a}$.

    Lastly, we explain the funcotriality. Let $\pi,\tau$ be two nuclear representations in $\Smod_P$ with a continuous $P$-equivariant map $\varphi:\pi\lra \tau$. Then it will induce two $P$-equivariant continuous maps
    \[
   \varphi_o: \CS(\CO)\cdot \pi\lra \CS(\CO)\cdot \tau\text{ and } \varphi_s: \SInd_S^P(\Psi(\pi))\lra \SInd_S^P(\Psi(\tau)) 
    \]
    such that the following commutative diagram holds
    \begin{center}
        \begin{tikzcd}
            \CS(\CO)\cdot \pi\ar[r,"\Xi"]\ar[d,"\varphi_o"] & \SInd_S^P(\Psi(\pi)) \ar[d,"\varphi_s"] \\
             \CS(\CO)\cdot \tau\ar[r,"\Xi"] & \SInd_S^P(\Psi(\tau)) .
        \end{tikzcd}
    \end{center}
    This is deduced from the functoriality of the map $\widetilde{a}$.
\end{proof}
\begin{remark}[Du Cloux's functor vs. twisted Jacquet functor]
    In the setting of Proposition~\ref{open-iso}, Du Cloux defines an analogue of the twisted Jacquet functor in~\cite[Lemma 2.5.7]{Fd}, which is automatically Hausdorff. For a long time, it has been believed that these two functors are identical, and this would imply that the twisted Jacquet functor is Hausdorff. 

    We provide evidence for this insight: when the (unnormalized) twisted Jacquet functor is known to be Hausdorff, we can prove that the two functors coincide. Let 
\[
\widetilde{\CS(\CO)} := \CS(\CO) \oplus \BC e
\]
be the algebra obtained by adjoining an identity element $e$.  Du Cloux's functor is defined as
\[
 \pi(\phi):= \pi/(\mathfrak{m}_{\phi}\cdot\pi),
\]
where $\mathfrak{m}_{\phi}:=\{(\varphi,-\varphi(\phi))\in\widetilde{\CS(\CO)} \mid \varphi\in \CS(\CO)\}$. The element in $\widetilde{\varphi}=(\varphi,-\varphi(\phi))\in\mathfrak{m}_{\phi}$ is determined by the $\varphi\in\CS(\CO)$. Thus, we will simply use $\varphi$ to denote $\widetilde{\varphi}$.

On the one hand, we can prove $\mathfrak{m}_{\phi}\cdot\pi \subset N(\pi\otimes \phi^{-1})$. Let $f\in \CS(N)$ such that $\widehat{f}\in \CS(\CO)$. Then
\[
\widehat{f}\cdot v = \int_N f(n) (\pi(n)-\phi(n))v \, dn \in N(\pi\otimes \phi^{-1})
\]
when $N(\pi\otimes \phi^{-1})$ is closed in $\pi$. For the reverse inclusion, we use the description of $\mathfrak{m}_{\phi}\cdot\pi$ in the proof of~\cite[Lemma 2.5.7]{Fd}. According to his proof, there exists a Nash open neighborhood of $\phi$ called $Y$, and a Nash open neighborhood $\mathscr{U}$ of $e\in P$ which is invariant under the left action of $S$ such that the restriction of the orbit map 
\[
  P \lra \CO \quad p\mapsto \phi\cdot p
\] 
to $\mathscr{U}$ is a Nash principal $S$-bundle over $Y$. Then there is an $\CS(Y)$-module isomorphism
\[
\CS(Y)\cdot \pi \stackrel{\beta}{\simeq} \SInd_{S}^{\mathscr{U}}(F)
\]
for some $F\in \Smod_S$. Fix an $\alpha\in \CC_c^{\infty}(Y)$ such that $\alpha(\xi)=1$ around $\phi$, then part (d) of the proof shows that
\begin{equation}\label{description du cloux}
     \mathfrak{m}_{\phi}\cdot\pi = \{v\in \pi\mid \beta(\alpha\cdot v)(1)=0\}.
\end{equation}
We observe that 
\[
\fkn(\pi\otimes (-\phi))= (\fkn (\CS(N)\otimes (-\phi)))\cdot \pi= \mathfrak{I}_{\phi}\cdot \pi ,
\]
where $\mathfrak{I}_{\phi}:= \{\varphi\in \CS(\widehat{N})\mid \varphi(\phi)=0\}$. It is not hard to see that $\mathfrak{I}_{\phi}\cdot \pi$ is contained in the right hand side of~\eqref{description du cloux}.

\end{remark}
This proposition has a corollary on the exactness of twisted Jacquet functor.
\begin{corollary}\label{open-exa}
    Under the same assumption as Proposition~\ref{open-iso}, if 
    \[
    0 \lra \pi_1 \lra \pi_2\lra \pi_3\lra 0
    \]
    is a short exact sequence in $\Smod_P$, and $\CS(\CO)\cdot \pi_1$ is closed in $\pi_1$, then $\Psi_{\phi}$ is exact. Namely, we have a short exact sequence
    \[
     0 \lra \Psi_{\phi}(\pi_1) \lra \Psi_{\phi}(\pi_2)\lra \Psi_{\phi}(\pi_3)\lra 0.
    \]
\end{corollary}
\begin{proof}
    By Proposition~\ref{open-iso} and Lemma~\ref{exa_lem}, the following sequence is exact:
    \[
    0 \lra \SInd_{S_{\phi}}^P( \Psi_{\phi}(\pi_1)) \lra  \SInd_{S_{\phi}}^P( \Psi_{\phi}(\pi_2))\lra  \SInd_{S_{\phi}}^P( \Psi_{\phi}(\pi_3))\lra 0.
    \]
    The argument in~\cite[Corollary 7.6]{WZ} shows that
    \[
    \rmh_i(\fkn, \SInd_{S_{\phi}}^P( \Psi_{\phi}(\pi_1))\otimes (-\phi))=\begin{cases}
        \Psi_{\phi}(\pi_1), \quad &\text{ if } i=0\\
        0, \quad &\text{ if } i\neq 0.
    \end{cases}
    \]
    This completes the proof.
\end{proof}

\subsubsection{Proof of Theorem~\ref{Canonical filtration}} 
In this subsection, we prove the Theorem~\ref{Canonical filtration} in full generality. Let $\pi$ be a Casselman-Wallach representation of $G$. By Casselman's embedding theorem, there is a generalized principal series $I$ and a closed embedding $\pi\hookrightarrow I$. By Lemma~\ref{generalized prin lem}, for any integer $\ell$, $\CS(\widehat{N}\setminus Z_{\ell})\cdot I$ is closed in $I$. Therefore, 
\[\left(\CS(\widehat{N}\setminus Z_{\ell})\cdot I\right)\bigcap \pi \text{ is closed in } \CS(\widehat{N}\setminus Z_{\ell})\cdot I. \]
Since $\CS(\widehat{N}\setminus Z_{\ell})\cdot I$ is a differentiable $\CS(\widehat{N}\setminus Z_{\ell})$-module and $\CS(\widehat{N}\setminus Z_{\ell})$ is hereditary, we have
\[
 \CS(\widehat{N}\setminus Z_{\ell})\cdot \left((\CS(\widehat{N}\setminus Z_{\ell})\cdot I)\bigcap \pi\right)= \left(\CS(\widehat{N}\setminus Z_{\ell})\cdot I\right)\bigcap \pi.
\]
On the other hand, a priori,
\[
\left(\CS(\widehat{N}\setminus Z_{\ell})\cdot I\right)\bigcap \pi \supset \CS(\widehat{N}\setminus Z_{\ell})\cdot\pi\supset \CS(\widehat{N}\setminus Z_{\ell})\cdot \left((\CS(\widehat{N}\setminus Z_{\ell})\cdot I)\bigcap \pi\right).
\]
Consequently, 
\[
\left(\CS(\widehat{N}\setminus Z_{\ell})\cdot I\right)\bigcap \pi = \CS(\widehat{N}\setminus Z_{\ell})\cdot\pi,
\]
which implies that $\CS(\widehat{N}\setminus Z_{\ell})\cdot\pi$ is closed in $\pi$. 

Consequently, the statement (1) follows from Theorem~\ref{topological spectral short exa}, and the statement (2) follows from Proposition~\ref{open-iso}. 

\begin{remark}
    We want to mention that under the assumption $\spadesuit$, the exactness of Casselman-Jacquet functor is a direct consequence of Theorem~\ref{Canonical filtration} and Lemma~\ref{exa_lem}.
\end{remark}

\subsection{Mirabolic case}
In this subsection, we specialize to the case that $G=GL_n$ and $P=M_n$, the mirabolic subgroup. Recall that $M_n$ has a standard Levi decomposition $M_n=\GL_{n-1}\ltimes V_n$ such that $\widehat{V_n}$ consists of a unique open orbit and a unique closed orbit. The stabilizer of the standard character $\psi_n$ in the open orbit is $M_{n-1}$. The Mackey induction of the open orbit is simply denoted as $I(\cdot)$ and the Mackey induction of the closed orbit is denoted as $E(\cdot)$.

We first recall various derivative functors that will be involved in the canonical Bernstein-Zelevinsky filtration. Define the absolute value for Archimedean local field as $|x|_{\mathbb{R}}=|x|$ for $x\in \mathbb{R}$, while $|x|_{\mathbb{C}}=|x|^2$ for $x\in \mathbb{C}$.
\begin{definition}\label{def-der-gl}
    Let $\sigma$ be a smooth moderate growth Fr\'echet representation of $M_n$, we define
    $$\Psi(\sigma):=|\det|_{\mathbb{K}}^{-1/2} \otimes \sigma/\Span\{\alpha v -\psi_n(\alpha)v \mid v\in\sigma, \alpha\in \fkv_n\}$$
and
$$\Phi(\sigma):= \varprojlim_{l} \sigma /\Span\{\kappa v\mid v  \in \sigma ,\kappa \in ({\mathfrak{v}}_{n})^{\otimes l} \}$$
For a positive integer $k$, the \textbf{$k$-th derivative} of $\sigma$ is defined to be $D^k(\sigma):=\Phi\circ\Psi^{k-1}(\sigma)$. The \textbf{depth} of representation $\sigma$ is the maximal positive integer $k_0$ such that $D^{k_0}(\sigma)\neq 0$, and $D^{k_0}(\sigma)$ is called the \textbf{highest derivative} of $\sigma$, denoted by $\sigma^{-}$. The corresponding opposite functors $\ov{\Psi}$, $\ov{\Phi}$, and $\ov{D}^k$ are defined via the transpose of the Lie algebra and the corresponding character.
\end{definition}
We recall that for a Casselman-Wallach representation $\pi$ of $\GL_n$, $\pi|_{M_n}$ admits a Bernstein-Zelevinsky filtration
\begin{align*}
     & \pi|_{M_n}=\sigma_0\supset  \sigma_1\supset \dots \supset \sigma_m \text{ with }\\
      \sigma_i=&\sigma_{i,0}\supset \sigma_{i,1}\supset \dots \supset \sigma_{i,j} \supset \dots \supset \sigma_{i+1},
\end{align*}
which is finer than coarse spectral filtration; see~\cite[Theorem 1.1 and Definition 3.2]{WZ} for details. Moreover, $\fkv^l_n\pi$ is closed in $\pi$ for any integer $l$ by~\cite[Theorem 1.4]{WZ}. Thus, $D^{k+1}$ is a functor $\mathrm{CW}_{\GL_n}\to \Smod_{M_{n-k}}$ for any non-negative integer $k$.

We first record a lemma which is useful in our canonical filtration Theorem~\ref{canon-fil-GL_n}.
\begin{lemma}\label{BZ-fil-psi-lem}
    The datum $\{\Psi^k(\sigma_i),\Psi^k(\sigma_{i,j})\}$ consists of a Bernstein-Zelevinsky filtration on $\Psi^k(\pi)$ for any integer $k$.
\end{lemma}
\begin{proof}
    We observe that each $\sigma_{i,j}$ and $\pi/\sigma_{i,j}$ inherit a BZ-filtration from $\pi$. Hence, we get a short exact sequence
    \[
    0\lra \Psi^k(\sigma_{i,j})\lra \Psi^k(\pi)\lra \Psi^k(\pi/\sigma_{i,j})\lra 0,
    \]
    where $\Psi^k$ is exact and $\Psi^k(\sigma_{i,j})$ is Hausdorff by~\cite[Lemma 2.17, Proposition 4.1]{WZ}. In other words, $\Psi^k(\sigma_{i,j})$ is a closed subspace of $\Psi^k(\pi)$. Moreover, for a successive quotient isomorphic to $I^{d}E(\tau)$ in the BZ-filtration of $\pi$, where $\tau$ is an irreducible representation of $\GL_{n-d-1}$, we have
    \[
    \Psi^k(I^dE(\tau))\simeq \begin{cases}
        I^{d-k}E(\tau) & d\geq k\\
        0 & d<k
    \end{cases}
    \]
    by~\cite[Proposition 4.1]{WZ} again. Consequently, we obtain the topological isomorphism
    \[
    \Psi^k(\sigma_i)/\Psi^k(\sigma_{i+1})\simeq \Psi^k(\sigma_i/\sigma_{i+1})\simeq \Psi^k(\mathop{\varprojlim}\limits_j \sigma_i/\sigma_{i,j}) \simeq \mathop{\varprojlim}\limits_j\Psi^k(\sigma_i)/\Psi^k(\sigma_{i,j})
    \]
    by~\cite[Lemma 2.10]{WZ} and the open mapping theorem.
 \end{proof}
 Lemma~\ref{BZ-fil-psi-lem} has a direct corollary, which also provides an affirmative answer to the open question~\cite[3.1, (3)]{AGS15a}.
\begin{corollary}\label{derivative exact}
    The derivative functor
    \[
    D^{k+1}:\mathrm{CW}_{\GL_n}\lra \Smod_{M_{n-k}}
    \]
    is an exact functor for any non-negative integer $k$.
\end{corollary}
\begin{proof}
     By Lemma~\ref{BZ-fil-psi-lem}, for any Casselman-Wallach representation $\pi$, $\Psi^k(\pi)$ has a BZ-filtration. Hence, Theorem~\ref{Canonical filtration} shows that there is a short exact sequence
    \[
    0\lra \CS(\widehat{V_{n-k}}\setminus\{0\})\cdot \Psi^k(\pi)\lra \Psi^k(\pi )\lra D^{k+1} (\pi)\lra 0.
    \]
    Consequently, this corollary follows from Lemma~\ref{exa_lem} since $\Psi^k$ is exact.
\end{proof}

 Our next theorem is an Archimedean counterpart of the filtration appearing in~\cite{BZ77}.
\begin{theorem}\label{canon-fil-GL_n}
    Let $\pi$ be a Casselman-Wallach representation of $\GL_n$. Then $\pi|_{M_n}$ has a decreasing filtration
    \[
    \pi|_{M_n}=\pi_0\supset \pi_1\supset \dots\supset \pi_n= 0
    \]
    such that
    \[
    \pi_k/\pi_{k+1}\simeq I^{k} ( D^{k+1}(\pi)) \text{ for}\quad 0\leq k\leq n-1.
    \]
\end{theorem}

\begin{proof}
    Let $\pi_1=\CS(\widehat{V_n}\setminus \{0\})\cdot \pi\simeq I(\Psi(\pi))$. Then by Theorem~\ref{Canonical filtration}, we have a short exact sequence
    \[
    0\lra \pi_1\lra \pi|_{M_n} \lra D^1(\pi) \lra 0.
    \]
    In what follows, we inductively demonstrate a filtration of $\pi_1$. More precisely, we inductively realize $\pi_k:=I^{k}(\Psi^{k}(\pi))$ as a $M_n$-subrepresentation of $\pi$. Suppose that we have done for some integer $k$. We proceed to prove for integer $k+1$. By Lemma~\ref{BZ-fil-psi-lem}, we can apply Theorem~\ref{Canonical filtration} to $\Psi^k(\pi)$ and obtain the short exact sequence
    \[
    0\lra  I(\Psi^{k+1}(\pi))\simeq \CS(V_{n-k}\setminus\{0\})\cdot \Psi^k(\pi) \lra \Psi^k(\pi) \lra  D^{k+1}(\pi)\lra 0.
    \]
    Consequently, applying the exact functor $I^k$ to above exact sequence, we get
    \[
    0\lra \pi_{k+1}\lra \pi_k\lra I^kD^{k+1}(\pi)\lra 0,
    \]
    which finishes the proof.
\end{proof}

As already observed in the proof of~\cite[Theorem 1.6]{WZ}, it is helpful to consider the canonical filtration of $\pi|_{\ov{M_n}}$. By~\cite[Proposition 3.17]{WZ}, for any Casselman-Wallach representation $\pi$ of $\GL_n$, $\pi|_{\ov{M_n}}$ has an opposite Bernstein-Zelevinsky filtration
\begin{align*}
     & \pi|_{\ov{M_n}}=\ov{\sigma_0}\supset  \ov{\sigma_1}\supset \dots \supset \ov{\sigma_m} \text{ with }\\
      \ov{\sigma}_i=&\ov{\sigma_{i,0}}\supset \ov{\sigma_{i,1}}\supset \dots \supset \ov{\sigma_{i,j}} \supset \dots \supset \ov{\sigma_{i+1}}.
\end{align*} Consequently, the same argument as in Lemma~\ref{BZ-fil-psi-lem} leads to the following result.
\begin{lemma}
    The datum $\{\ov{\Psi}^k(\ov{\sigma_i}),\ov{\Psi}^k(\ov{\sigma_{i,j}})\}$ consists of a Bernstein-Zelevinsky filtration on $\ov{\Psi}^k(\pi)$ for any integer $k$.
\end{lemma} 

On the other hand, by Theorem~\cite[Theorem 1.3]{WZ}, we know that $L^i\ov{B}^k(\pi)$ is Casselman-Wallach for any integer $i$ and $k$, which implies that
\[
\ov{\Psi}^k(\pi)/ \ov{\fkv_{n-k}}^{j}\ov{\Psi}^k(\pi)
\]
is Hausdorff for any integer $k,j$. Consequently, the argument in Theorem~\ref{canon-fil-GL_n} provides the proof for the opposite canonical filtration.
\begin{theorem}
     Let $\pi$ be a Casselman-Wallach representation of $\GL_n$. Then $\pi|_{\ov{M_n}}$ has a decreasing filtration
    \[
    \pi|_{\ov{M_n}}=\ov{\pi_0}\supset \ov{\pi_1}\supset \dots\supset \ov{\pi_n}= 0
    \]
    such that
    \[
   \ov{\pi_k}/\ov{\pi_{k+1}}\simeq \ov{I}^{k} ( \ov{D}^{k+1}(\pi)) \text{ for} \quad 0\leq k\leq n-1.
    \]
\end{theorem}

Let $j$ be the maximal integer such that the derivative $\ov{D}^j(\pi)\neq 0$. Then $\ov{D}^j(\pi)$ is called the left highest derivative and is denoted by ${}^-\pi$. We provide a remark on the highest derivative. 
\begin{remark}
    Let $\pi$ be Casselman-Wallach representation of depth $d$. By~\cite[Theorem A]{AGS15a} and~\cite{GS15}, we know that $d$ is the minimal positive integer $d$ such that $X^d=0$ for any $X\in \mathfrak{gl}_n^*\simeq \mathfrak{gl}_n$ in the annihilated variety of $\pi$. Moreover, $D^d(\pi)=\Psi^{d-1}(\pi)$, which is a Casselman-Wallach representation of $\GL_{n-d}$. By a similar argument, one can show that $d$ is the maximal integer that $\ov{D}^d(\pi)\neq 0$ and $\ov{D}^d(\pi)=\ov{\Psi}^{d-1}(\pi)$, which is a Casselman-Wallach representation of $\GL_{n-d}$. 
\end{remark}

\subsection{The highest derivative of contragredient representations}
We first introduce a notion of bilinear pairing following~\cite{BZ77}.
\begin{definition}
    Let $\sigma_1,\sigma_2$ be two representations in $\Smod_G$, where $G$ is an almost linear Nash group. Define the vector space of bilinear pairings $\Bil(\sigma_1,\sigma_2)$, which consists of $G$-equivariant continuous bilinear maps
    \[
    B: \sigma_1\times \sigma_2 \lra \delta_G^{-1},
    \]
    where $G$ diagonally acts on the left hand side. We say a bilinear pairing $B$ is \textbf{non-degenerate} with respect to $\sigma_1$ if for any non-zero $v_1\in\sigma_1$, there is some $v_2\in\sigma_2$ such that $B(v_1,v_2)\neq 0$.
\end{definition}

 The following lemma concerns the bilinear pairing between Mackey induction and trivial extension.
\begin{lemma}\label{bil-lem}
   Let $\gamma\in\Smod_{M_n}$ such that $\CS(\widehat{V_n}\setminus\{0\})\cdot \gamma=0$,  and let $\sigma, \sigma^{\flat}\in\Smod_{M_{n-1}}$. Then
    \begin{enumerate}
        \item $\Bil(I(\sigma), I(\sigma^{\flat}))\simeq \Bil (\sigma,\sigma^{\flat})$, which preserves non-degeneracy;
        \item $\Bil(I(\sigma),\gamma)=0$.
    \end{enumerate}
\end{lemma}

\begin{proof}
    (2) follows from the fact that $I(\sigma)$ is a differentiable $\CS(\widehat{V_n}\setminus\{0\})$-module and the fact that if $B\in \Bil(I(\sigma),\gamma)$, then
    \[
    B(f\cdot x_1,x_2)=B(x_1, \tilde{f}\cdot x_2),
    \]
    where $f,\tilde{f}\in \CS(\widehat{V_n}\setminus\{0\})$ and $\tilde{f}(\xi)=f(-\xi)$ for $\xi\in \widehat{V_n}\setminus\{0\}$.

    We proceed to prove part (1). Firstly, there is an injective linear map
    \[
    \iota:\Bil (\sigma,\sigma^{\flat}) \lra \Bil(I(\sigma), I(\sigma^{\flat}))
    \]
    defined as follows. To a bilinear pairing $B\in \Bil (\sigma,\sigma^{\flat})$, we associate the bilinear pairing
    \[
    (f_1,f_2)\mapsto \int_{H_{n,2}\backslash M_n} B(f_1(m),f_2(m))dm,
    \]
    where $f_1\in I(\sigma)$ and $f_2\in I(\sigma^{\flat})$. Here, we define the Mackey induction $I(\sigma^{\flat})$ through the unitary character $\psi_n^{-1}$. Then we define an injective inverse map $\epsilon$. Note that there is an injective map
    \[
    \Bil(I(\sigma), I(\sigma^{\flat}))\lra \Hom_{M_n}(I(\sigma)\delta_{M_n}, I(\sigma^{\flat})').
    \]
    Let $\Gamma\in \Hom_{M_n}(I(\sigma)\delta_{M_n}, I(\sigma^{\flat})')$. Then 
    \[
    \mathrm{Im} \ \Gamma\subseteq \CS(M_n)\cdot I(\sigma^{\flat})' \subseteq  {}^{\infty}\Ind_{M_{n-1}\ltimes V_n}^{M_n}((\sigma^{\flat})'\otimes \psi_n),
    \]
    where the last inclusion follows from~\cite[Chapter 6, Theorem 19]{Sch}. On the other hand, since $I(\sigma)$ is a differentiable $\CS(\widehat{V_n}\setminus\{0\})$-module, 
    \[
    \mathrm{Im} \ \Gamma\subseteq\SInd_{M_{n-1}\ltimes V_n}^{M_n}((\sigma^{\flat})'\otimes\psi_n).
    \]
    Therefore, by Frobenius reciprocity, $\Gamma$ induces a linear map
    \[
   I(\sigma)\delta_{M_n} \lra \sigma\delta_{M_{n-1}} ^2\stackrel{\widehat{\Gamma}}{\lra}  (\sigma^{\flat})'\delta_{M_{n-1}},
    \]
    where the first map is evaluation at the identity point.  It suffices to show that $\widehat{\Gamma}$ is continuous. 

    Suppose that a sequence $\{t_i\} \subset \sigma \delta_{M_{n-1}}$ converges to zero. Choose a function $f \in \CC^{\infty}_c(T_{n-1} \cdot \ov{U_{n-2,1}})$ that is constant on an open neighborhood $\mathscr{U}$ of the identity. Here, $T_{n-1}$ is the rank-$1$ torus in the bottom right corner of $\GL_{n-1}$, and $T_{n-1} \cdot \ov{U_{n-2,1}}$ is an affine open neighborhood of the identity that trivializes the Mackey induction. Then $f_i := f \cdot t_i \in I(\sigma)\delta_{M_n}$ is a sequence converging to zero in $I(\sigma)\delta_{M_n}$. Moreover, for any $X \in \U(\fkt_{n-1} + \ov{\fku_{n-2,1}})$, we have
    \[
    X\cdot \Gamma(f_i) =\Gamma(X\cdot f_i),
    \]
   which implies that $\Gamma(f_i)$ is also constant on $\mathscr{U}$.

Now assume, for contradiction, that $\widehat{\Gamma}(t_i)$ does not converge. By definition, there exists a bounded set $\mathscr{S} \subset \sigma^{\flat}$ such that
    \[
    \sup_{s\in\mathscr{S}} |\widehat{\Gamma}(t_i)(s)|
    \]
 does not converge to zero. Define a bounded set
    \[
    \mathscr{T}:=\{\ov{f} \cdot s| s\in \mathscr{S}\}\subset I(\sigma^{\flat}).
    \]
   Then $\Gamma(f_i)$ fails to converge with respect to the semi-norm induced by $\mathscr{T}$, which contradicts the convergence of $f_i$.

Consequently, the inverse map $\epsilon$ is well-defined. The desired statement follows from the identity $\epsilon \circ \iota = \id$.
\end{proof}

\begin{proposition}\label{pairing-decent}
    Let $\sigma_1, \sigma_2 \in \Smod_{M_n}$ be representations that admit a canonical Bernstein–Zelevinsky filtration. If there is a non-degenerate pairing $B$ with respect to $\sigma_1$ between $\sigma_1$ and $\sigma_2$, then there is a non-degenerate pairing with respect to $D^d(\sigma_1)$ between $D^d(\sigma_1)$ and $D^d(\sigma_2)$, where $d$ is the depth of $\sigma_1$.
\end{proposition}
\begin{proof}
    We argue by induction on $d$. Suppose $d=1$, then $\sigma= D^1(\sigma)$. Then the result follows from part (2) of Lemma~\ref{bil-lem}. Suppose $d> 1$, then $B$ will induce a non-degenerate pairing between $I(\Psi(\sigma_1))$ and $I(\Psi(\sigma_2))$. Hence, by part (1) of Lemma~\ref{bil-lem}, there is a non-degenerate pairing between $\Psi(\sigma_1)$ and $\Psi(\sigma_2)$. Then the result follows from the inductive hypothesis on $d-1$.
\end{proof}
In particular, this proposition directly implies the following corollary.
\begin{corollary}\label{pairing-coro}
    Let $\pi$ be a Casselman-Wallach representation of $\GL_n$ with depth $d$. Then there is a non-degenerate $M_{n-d+1}$-equivariant bilinear map
    \[
    \pi^{-}\times  (\pi^{\vee})^- \lra \BC.
    \]
\end{corollary}

\subsection{Conjectures on highest derivatives}\label{conj sec}
Let $\pi$ be a Casselman-Wallach representation of $\GL_n$. Let 
\[ 
 F_0\pi^-=0\subset F_1\pi^- \subset \dots \subset F_p\pi^- =\pi^-
\] 
be the increasing socle filtration on $\pi^-$ as a representation of $M_{n-d+1}$, where $d$ is the depth of $\pi$. That is, $F_{k+1}\pi^-/F_k\pi^-$ is the maximal semisimple submodule of $\pi^-/F_k\pi^-$.
Moreover, let \[ 
 G^0\pi^-=\pi^- \supset G^1\pi^- \supset \dots \supset G^q\pi^- =0
\] 
be the decreasing co-socle filtration on $\pi^-$. In other words, $G^k\pi^-/G^{k+1}\pi^-$ is the maximal semisimple quotient of $G^k\pi^-$. 

Corollary~\ref{pairing-coro} shows that there is a non-degenerate pairing between $F_{k+1} \pi^-/F_{k} \pi^-$ and $G^{p-k}(\pi^{\vee})^-/ G^{p-k+1} (\pi^{\vee})^-$ as representations of $\GL_{n-d}$. We propose the following conjecture concerning the structure of (co)-coscole filtration.
\begin{conjecture}\label{socle conjecture}
Let $\pi$ be an irreducible representation of $\GL_n$.
   \begin{enumerate}
       \item The socle filtration coincides with the co-socle filtration. Namely, $p=q$ and
        \[
        F_k \pi^- =G^{p-k} \pi^-
        \]
        for any integer $k$.
        \item The length of socle filtration is symmetric. That is, $\mathrm{length} (F_k\pi^- /F_{k-1}\pi^-)=\mathrm{length}(F_{p-k+1}\pi^- /F_{p-k}\pi^-)$ for any integer $k$.
        \item $\pi^-$ has a unique irreducible quotient (thus unique irreducible submodule). In other word, $\mathrm{length} (F_1\pi^- )=1$.
   \end{enumerate} 
\end{conjecture}

\begin{remark}
    By MVW-involution, we have an identification of vector spaces ${}^-\pi\simeq (\pi^{\vee})^-$ such that $(\pi^{\vee})^-(g)={}^-\pi(g^{-t})$ for $g\in M_{n-d+1}$. Thus, 
    \[
  \left(  F_k({}^-\pi)/F_{k-1}({}^-\pi) \right)^{\vee}\simeq  F_k(\pi^{\vee})^-/F_{k-1}(\pi^{\vee})^-.
    \]
    Moreover, if (1) holds for any irreducible representation, (1) also holds for the left highest derivative. On the other hand, by Corollary~\ref{pairing-coro}, 
    \[
    F_k(\pi^{\vee})^-/F_{k-1}(\pi^{\vee})^-  \simeq \left( G^{p-k}(\pi^-) / G^{p-k+1}(\pi^-)\right)^{\vee}.
    \]
    In other words, the left socle filtration coincides with the right co-socle filtration, and vise versa. Moreover, if (1) holds, then (2) is equivalent to 
    \[
    \mathrm{length} (F_k\pi^- /F_{k-1}\pi^-)= \mathrm{length} (F_k({}^-\pi) /F_{k-1}({}^-\pi) ),
    \]
    which should be easier to prove.
\end{remark}

\begin{example}\label{finite dim exa}
    The conjecture~\ref{socle conjecture} holds for any irreducible finite dimensional representations of $\GL_n(\BR)$. At this time, $\pi^-=\pi|_{M_n}$. Since the character of $\GL_n(\BR)$ does not influence the filtration, we assume $\pi$ comes from a holomorphic representation of $\GL_n(\BC)$. We identify such irreducible representation with an $n$-tuple of non-increasing integers via the highest weight. By Weyl's branching law, the restriction of $\pi=\lambda=(a_1,\cdots,a_n)$ to $\GL_{n-1}$ is multiplicity-free, and has the irreducible component $\mu=(b_1,\dots ,b_{n-1})$ if and only if $\mu$ interlaces $\lambda$, that is
\[
 a_1\geq b_1\geq a_2\geq b_2\geq \dots\geq b_{n-1}\geq a_n,
\]
and is denoted by $\mu\prec \lambda$. Let $E_k$ be the subspace of $\pi|_{\GL_{n-1}}$ consisting of 
\[
 \bigoplus_{\mu\prec\lambda} \mu, \text{ runs over } |\mu|= -k-a_n +|\lambda|,
\]
where $|\lambda|=\sum_{i=1}^n a_i$. Thus, $\fkv_n \cdot E_{k}\subset E_{k-1}$ by considering the weight of the center of $\GL_{n-1}$. On the other hand,~\cite[Theorem 1.5]{Ne18} shows that $\fkv_n \cdot E_{k}= E_{k-1}$ and $\fkv_n \cdot \mu\neq 0$ for $\mu\subset E_k$ and $k\neq 0$. Thus, since any irreducible finite-dimensional representation of $M_n$ is killed by $\fkv_n$, we have $p=q=a_1-a_n+1$ and
\[
 F_k\pi= G^{p-k}\pi= \bigoplus_{i=0}^{k-1} E_i.
\]
Part (2) of the conjecture is obvious.  
\end{example}

\begin{example}
    Let $\pi\simeq \chi_1\times \dots \times \chi_k$ be an irreducible monomial, where $\chi_i$ is a character of $\GL_{n_i}$ such that $n=\sum_{i=1}^k n_i$. Then by~\cite{AGS15a,AGS15b},
    \[
    \pi^-\simeq \chi_1|_{\GL_{m_1}}\times \dots \times \chi_k|_{\GL_{m_k}},
    \]
    where $m_i=n_i-1$ and the monomial of $\GL_{n-k}$ is inflated to a representation of $M_{n-k+1}$.  Then by~\cite[Theorem 1]{G17}, $\pi^-$ is irreducible as well. Hence, the conjecture formally holds.
\end{example}

\begin{example}
    Let $\pi\simeq \gamma \times \chi$, where $(\gamma,W)$ is an irreducible finite-dimensional representation of $\GL_{n-1}$ and $\chi$ is a character of $\GL_1$. Then $\mathrm{depth}(\pi)=2$. $M_n$-action on $P_{n-1,1}\backslash \GL_n$ has a unique open orbit and a unique closed orbit, which gives 
    \[
     0 \lra \pi_o\lra \pi|_{M_n} \lra  \pi_c\lra 0.
    \]
    Moreover, $\Psi(\pi)=\Psi(\pi_o)$. We realize $\pi_o$ as Schwartz sections over the tempered bundle 
    \[(\chi\boxtimes \gamma)\times_{\ov{P_{1,n-1}}\cap M_n} M_n\lra (\ov{P_{1,n-1}}\cap M_n)\backslash M_n .\]
    Let $U:= U_{1,n-1},V:= U_{1,n-1}\cap V_n$ and let $\widetilde{V}:= U_{1,n-2}\subset \GL_{n-1}$. Then the open subsets $U$ and $U\cdot w$, where $w$ is the permutation matrix in $\GL_{n-1}$ corresponding to $(1,2,\cdots,n-1)$, consist of an affine trivialization of the tempered bundle. We compute the $\fkv_n$-action on these trivializations. 
    \begin{enumerate}
        \item For affine chart $U$: Under the isomorphism
        \[
        \CS(U, \chi\boxtimes \gamma)\stackrel{\mathcal{PF}}{\simeq} \CS(\widetilde{V}\times V^*,\chi\boxtimes \gamma),
        \]
        we have 
        \[\Psi(\CS(U, \chi\boxtimes \gamma))\simeq \Psi(\CS(\widetilde{V}\times (V^*\setminus\{0\}),\chi\boxtimes \gamma))
        \]
        by~\cite[section 3]{WZ}, where $\mathcal{PF}$ is the partial Fourier transform along $V^*$. We use $M$ to denote $\widetilde{V}\times (V^*\setminus\{0\})$ for simplicity. Then $\fkv_n$-action on $\CS(W,\chi\boxtimes \gamma)$ is given by 
        \[
        \epsilon: M\lra \Lie(V_n)^* , (a,\xi)\mapsto d\psi(1)(\xi, a\xi),
        \]
        where $a\in \BK^{n-2}$ is the standard coordinate of $\widetilde{V}$ and $0\neq\xi \in V^* $, and 
        \[
       \widetilde{\gamma}: \fkv_n \lra \fkv_n/\fkv\simeq \fkv_{n-1} \stackrel{\gamma}{\lra} \mathfrak{gl}(W)
        \]
        via equation~\eqref{action eq} in Proposition~\ref{twisted Jacquet computation}. Since $\epsilon^{-1}(d\psi(1)(0,\cdots,0,1))=\emptyset$, by Proposition~\ref{twisted Jacquet computation}, we have
        $\Psi(\CS(M, \chi\boxtimes \gamma))=0$.
        \item For affine chart $U\cdot w$: We identify $U$ with $U\cdot w$ by right multiplying $w$, and carry out the partial Fourier transform on $V$ similarly. However, the $\fkv_n$-action is different. It is given by
        \[
        \epsilon: M\lra \Lie(V_n)^* , (a,\xi)\mapsto d\psi(1)(a\xi,\xi)
        \]
        and 
        \[
         \widetilde{\gamma}: \fkv_n \lra \fkv_n/\Ad_w(\fkv)\simeq \fkv_{n-1} \stackrel{\gamma}{\lra} \mathfrak{gl}(W).
        \]
        Since $\epsilon^{-1}(d\psi(1)(0,\cdots,0,1))=(0,\cdots,0,1)$, we have a natural isomorphism defined in Proposition~\ref{twisted Jacquet computation}:
        \[
      \varphi: \Psi( \CS(U\cdot w, \chi\boxtimes \gamma))\simeq W.
        \]
        Moreover, under this isomorphism, the action of $M_{n-1}$ on $W$ is identical to $\gamma|_{M_{n-1}}$. 
    \end{enumerate}
    Since twisted Jacquet functor is right exact, we obtain
    \[
    \Psi(\pi)\simeq \Psi( \CS(U\cdot w, \chi\boxtimes \gamma))\simeq \gamma|_{M_{n-1}}
    \]
    by a covering argument as in~\cite[Corollary 4.3]{WZ}. Therefore, the Conjecture~\ref{socle conjecture} follows from Example~\ref{finite dim exa}.
\end{example}

\subsection{Indecomposability of mirabolic restriction}\label{inde sec}
 In this subsection, we focus on the indecomposability of this restriction. Following~\cite{Ze80}, we first introduce a stronger notion of indecomposability.
\begin{definition}
Let $\sigma \in \Smod_{M_n}$. Then $\sigma$ is called strongly indecomposable if $\sigma_1 \cap \sigma_2 \neq 0$ for any two non-zero subrepresentations $\sigma_1, \sigma_2$.
\end{definition}

\begin{definition}\label{homo def}
 Let $\pi$ be a Casselman-Wallach representation of $\GL_n$. It is called \textbf{homogeneous} if any $M_n$-subrepresentation $\sigma$ satisfies
 \[
    \sigma \bigcap I^{d-1}(D^d(\pi))\neq 0,
 \]
 where $d$ is the depth of $\pi$.
\end{definition}

In $p$-adic case, \cite[Corollary 6.8]{Ze80} shows that any irreducible representation is homogeneous. We conjecture that it also holds for any irreducible representation in the archimedean case. In this section, we will prove this in some specific cases. 
\subsubsection{Irreducible unitary representations}
We first recall the Kirillov conjecture which considers the indecomposability of restriction of Hilbert representations.
\begin{theorem}[Kirillov conjecture]
    Let $\Pi$ be an irreducible unitary Hilbert representation of $\GL_n$. Then $\Pi|_{M_n}$ is still irreducible.
\end{theorem}
Initiated by \cite{Sa}, this conjecture has been studied for a long time. In \cite{Bar}, it is completely solved by the distribution method. By Mackey induction theory, there is an irreducible unitary Hilbert representation $A(\Pi)$ of $\GL_{n-d}$ such that 
\[
 \Pi|_{M_n}\simeq  \Ind_{H_{n,d}}^{M_n}(A(\Pi)\otimes \psi_{n,d}),
\]
for some integer $d$. It is proved in \cite{AGS15a} that
\[
d=\mathrm{depth}(\pi) \text{ and } A(\Pi)^{\infty}\simeq \pi^-,
\]
where $\pi=\Pi^{\infty}$. Consequently, we have a composition of continuous injective maps
\[
\Gamma:\pi|_{M_n}\hookrightarrow (\Pi|_{M_n})^{\infty}\hookrightarrow {}^{\infty}\Ind_{H_{n,d}}^{M_n}(\pi^{-}\otimes \psi_{n,d}).
\]
On the other hand, by Frobenius reciprocity, 
\[
\Hom_{M_n} (\pi|_{M_n}, {}^{\infty}\Ind_{H_{n,d}}^{M_n}(\pi^{-}\otimes \psi_{n,d}))\simeq \Hom _{H_{n,d}}(\pi, \pi^{-}\otimes \psi_{n,d}).
\]
Since $\pi^-$ is irreducible, $\Gamma$ corresponds to some scalar multiple of natural projection. Consequently, the injectivity of $\Gamma$ demonstrates that $\pi$ is homogeneous. This yields the following corollary.
\begin{corollary}
    Let $\pi$ be an irreducible unitary representation of $\GL_n$. Then $\pi|_{M_n}$ has a unique irreducible submodule $I^{d-1}(D^d(\pi))$, where $d$ is the depth of $\pi$. In particular, $\pi|_{M_n}$ is indecomposable.
\end{corollary}

\subsubsection{Irreducible generic representation}
We begin with a lemma that locates the irreducible generic representations in dominant principal series.
\begin{lemma}[Theorem 6.2,\cite{Vog78}]
Each dominant principal series of $\GL_n$ has a unique irreducible submodule, which is generic. Conversely, every irreducible generic representation arises as the unique irreducible submodule of some dominant principal series.
\end{lemma}

By the above lemma, it suffices to prove that dominant principal series is homogeneous. 
\begin{theorem}\label{homo-gen}
    Let $\pi=\chi_1\times \dots\times\chi_n$ be a principal series such that
    \[
    \chi_i |\det|_{\BK}(\det)^{r_1}(\ov{\det})^{r_2}\neq \chi_j
    \]
    for any $i<j$ and any non-negative integer $r_1,r_2$, then $\pi$ is homogeneous.
\end{theorem}

We introduce a lemma that will be used in the proof of Theorem~\ref{homo-gen}.
\begin{lemma}\label{open-sub}
    Let $\tau$ be a Casselman-Wallach representation of $\GL_{n-s}$, and let $\sigma\in\Smod_{M_s}$ which admits a BZ-filtration. Suppose $\omega$ is a non-zero $M_n$-subrepresentation of $\tau\mind \sigma$, then $\Psi(\omega)\neq 0$.
\end{lemma}
\begin{proof}
    By the assumption and argument in~\cite[section 3]{WZ}, we know that $\tau\mind \sigma$ admits a BZ-filtration. Consequently, $\omega$ also admits a BZ-filtration, and in particular, $\CS(\widehat{V_n}\setminus\{0\})\cdot \omega=0$ if $\Psi(\omega)=0$. Note that $\tau\mind \sigma$ is a subspace of $\CC^{\infty}(M_n, \tau\boxtimes \sigma)$ consisting of functions whose restriction to $V_n \cap U_{n-s,s}$ are Schwartz functions. Let $f\in \omega$ such that $f(e)\neq 0$. We write 
    \[
    V_n =(V_n \cap U_{n-s,s})\oplus (V_n\cap M_s)\simeq V_{n-s}\oplus V_s \quad v\mapsto (v_1,v_2)
    \]
    and 
    \[
    \widehat{V_n}\simeq \widehat{V_{n-s}}\oplus \widehat{V_s} \quad \phi\mapsto (\phi_1,\phi_2).
    \]
    Let $g_2\in \CS(\widehat{V_s})$ such that $g_2\cdot f(e)\neq 0$. Let $F$ be the non-zero Schwartz section on $V_{n-s}$ defined by $F(v_1):=g_2\cdot f(v_1)$. Hence $\widehat{F}\neq 0$, and in particular, there is some $\phi_1\neq 0$ such that $\widehat{F}(\phi_1)\neq 0$. Let $\ell$ be a continuous real functional on $\tau\boxtimes \sigma$ such that $\ell(\widehat{F}(\phi_1))> 0$. Consider an open neighborhood $\mathscr{U}$ of $\phi_1$ and a positive Schwartz function $g_1$ supported on $\mathscr{U}$ such that $0\notin \mathscr{U}$ and $\ell(\widehat{F}(\mathscr{U}))\subset \BR_{>0}$. Then for 
    \[g(\phi_1,\phi_2):=g_1(\phi_1)\cdot g_2(\phi_2)\in \CS(\widehat{V_n}\setminus\{0\})\],
    we have
    \[
   ( g\cdot f)(e)=\int_{\widehat{V_{n-s}}} g_1(\phi_1) \widehat{F}(\phi_1) d\phi_1 \neq 0.
    \]
    This leads to a contradiction.
\end{proof}
We are now ready to prove Theorem~\ref{homo-gen}.
\begin{proof}[Proof of Theorem~\ref{homo-gen}]
    We prove by induction on $n$. When $n=1$, the statement is trivial. We proceed to prove the statement assuming $n-1$ holds. Suppose $\omega$ is a non-zero $M_n$-submodule of $\pi|_{M_n}$ with $\mathrm{depth}(\omega)=d<n$. Write $\pi$ as $\chi_1\times \tau$, where $\tau=\chi_2\times \dots\times\chi_n$. Then the $M_n$-action on $P_{1,n-1}\backslash \GL_n$ gives rise to
    \[
    0\lra \pi_o \lra \pi|_{M_n}\lra \pi_c\lra 0.
    \]
    Considering $\omega\cap \pi_o$ and the image of $\omega$ in $\pi_c$, it suffices to deal with following two cases.
    \begin{itemize}
        \item \textbf{Case 1.} $\omega$ is a submodule of $\pi_o\simeq \tau\mind 1$, where $1$ is the trivial representation of the trivial group. By Lemma~\ref{open-sub}, we know that $\Psi(\omega)$, which is a submodule of $\Psi(\pi_o)\simeq \tau|_{M_{n-1}}$, is non-zero. Consequently, it contradicts to the induction hypothesis on $n$.
        \item \textbf{Case 2.} $\omega$ is a submodule of $\pi_c$. Let 
        \[
        \ov{\pi}:=|\det|_{\BK}^{-1} \pi^{\vee}= (|\det|_{\BK}\chi_1)^{-1}\times \dots\times (|\det|_{\BK}\chi_n)^{-1}.
        \]
        Then there is a non-generate bilinear pairing with respect to $\omega$ between $\omega$ and $\ov{\pi}|_{M_n}$. By Proposition~\ref{pairing-decent}, there is a non-generate bilinear pairing with respect to $D^d(\omega)$ between $D^d(\omega)$ and $D^d(\ov{\pi}|_{M_n})$. By explcit computation in~\cite[Theorem 3.12]{WZ}, we know that $D^d(\omega)$ has a filtration of $\GL_{n-d}$-representation, whose successive quotients has form
        \begin{equation}\label{side 1}
              \left(   |\det|_{\BK}^{1/2}(\det)^{r_1}(\ov{\det})^{r_2}\chi_1\right)\times \gamma 
        \end{equation}
        for some non-negative integer $r_1,r_2$ and some $\gamma\in \mathrm{CW}_{\GL_{n-d-1}}$. On the other hand, we know that $D^d(\ov{\pi}|_{M_n})$ has a filtration of $\GL_{n-d}$-representation, whose successive quotients has form
                \begin{equation}\label{side 2}
                       \lambda_1\times\dots\times \lambda_{n-d},
        \end{equation}
        where for each $i$, $\lambda_i\simeq \chi^{-1}_j|\det|_{\BK}^{-1/2}(\det)^{k_1}(\ov{\det})^{k_2} $ for some $j$ and some non-negative integer $k_1,k_2$. Consequently, by comparing infinitesimal characters, we observe that there is no non-generate pairing between any irreducible component in $\eqref{side 1}$ and irreducible component in $\eqref{side 2}$, which leads to a contradiction.
    \end{itemize}
\end{proof}

\begin{corollary}
   Let $\pi$ be an irreducible generic representation of $\GL_n$. Then $\pi|_{M_n}$ contains the Gelfand-Grev representation $I^{n-1}(\BC)$ as the unique irreducible submodule. In particular, $\pi|_{M_n}$ is indecomposable. 
\end{corollary}

The Gelfand-Kazhdan conjecture in~\cite{GK72} also follows from the theorem.
\begin{corollary}
    Let $\pi$ be an irreducible generic representation of $\GL_n$. Then there is a unique Kirillov model of $\pi$.
\end{corollary}

For general irreducible representations, we propose the following conjecture.
\begin{conjecture}\label{homo conj}
    Let $\pi$ be an irreducible representation of $\GL_n$. Then $\pi$ is homogeneous and $\pi|_{M_n}$ is indecomposable.
\end{conjecture}

\subsection{A counter-example}\label{counter-exa-sec}
In this subsection, we exhibit an example of $\pi\in\Smod_{M_n}$ such that the exact sequence in Theorem~\ref{Canonical filtration} does not hold. In \cite{CWYZ}, we conjecture that the exact sequence should hold under nuclear assumptions and provide some evidence. Here, we also show that the counter-example is not nuclear, which provides further evidence. We first recall a well-known result on irreducibility.
\begin{lemma}
    Let $G$ be a Lie group. If $E$ is an irreducible unitary Hilbert representation of $G$, then its smooth vector $E^{\infty}$ equipped with the topology defined in~\eqref{smooth vector top} is irreducible as well.
\end{lemma}
\begin{proof}
    Let $V$ be a non-zero closed sub-representation of $E^{\infty}$. Note that $V$ is dense in $E$ since $E$ is irreducible. Since the action map $a:\CC^{\infty}_c(G)\otimes E^{\infty}\lra E^{\infty}$ is surjective, it suffices to show that for any $f\in \CC^{\infty}_c(G)$ and $e\in E^{\infty}$, there exists a sequence of elements $\{v_n\}_{n=1}^{\infty}\subset V$ converging to $a(f\otimes e)$. 

    Let $\{w_n\}_{n=1}^{\infty}$ be a sequence of elements in $V$ which converges to $e$ in $E$. Then for any $X\in\U(\fkg)$,
    \[
    \lVert X(a(f\otimes w_n)-a(f\otimes e))\rVert \leq \lVert X(f)\rVert_{L^1} \cdot \lVert w_n-e\rVert .
    \]
    Hence, $v_n=a(f\otimes w_n)$ is a sequence we want.
\end{proof}

Let $E=\Ind_{V_2}^{M_2}(\psi)$ be the $L^2$-induction of the unitary character $\psi$. By Mackey induction theory, it is an irreducible Hilbert representation of $M_2$. Hence, its smooth vector $E^{\infty}$ equipped with the topology defined in~\eqref{smooth vector top} is an irreducible representation in $\Smod_{M_2}$. 

We realize $\Ind_{V_2}^{M_2}(\psi)$ as $L^2$-functions on $$\BR^{\times}\simeq \left\{ \begin{pmatrix}
    x & 0\\
    0 & 1\\
\end{pmatrix}\mid x\in\BR^{\times}\right\},$$
whose right invariant measure is $d\mu=\frac{dx}{|x|}$. Let $\CD$ be the subalgebra of the Nash differential operators on $\BR^{\times}$ generated by $x\frac{d}{dx}$ and polynomial functions $\BC [x]$. Then 
\[
E^{\infty}=\{f\in\CC^{\infty}(\BR^{\times})\mid Df\in L^2(\BR^{\times},d\mu), \forall D\in \CD\}.
\]
Now, we show that $\SInd_{V_2}^{M_2}(\psi)\subsetneq E^{\infty}\subsetneq \CO\Ind_{V_2}^{M_2}(\psi)$. 
\begin{enumerate}
    \item Let $\rho\in\CC^{\infty}(\BR)$ be a cutoff function such that $\supp(\rho)\subset [-2,2]$ and $\rho=1$ on $[-1,1]$. Then $\rho x\in E^{\infty}$ but $\rho x\notin \CS(\BR^{\times})$. 
    \item By the definition of tempered functions, it suffices to show that for any $f\in E^{\infty}$ and any positive integer $n$, there exists some positive integer $N$ such that
    \[
    |x^N \frac{d^n}{dx^n}(f)| \text{ is bounded.}
    \]
    Since $x\frac{d}{dx}(f)\in V^{\infty}$, it suffices to prove when $n=0$. We observe that 
    \[
    \frac{d}{dx}((x^mf)^2)=2m(x^mf)(x^{m-1}f)+2(x^mf)(x^m\frac{d}{dx}f)
    \]
 , and hence when $m\geq 1$
    \[
   \bigg |\frac{d}{dx}((x^mf)^2)\bigg|_{L^1} \leq 2m |x^mf|_{L^2}\cdot |x^{m-1}f|_{L^2}+ 2|x^mf|_{L^2}\cdot |x^m\frac{d}{dx}(f)|_{L^2}<+\infty.
    \]
   Here, the norm is with respect to the measure $dx$. Since $\big|\frac{d}{dx}((x^mf)^2)\big|_{L^1}$ is larger than the total variation of $(x^mf)^2$, which implies that $|x^mf|$ is bounded.
\end{enumerate}
Moreover, it is not hard to see that the restriction map in $x=1$ induces 
$$\Psi(E^{\infty})\simeq E^{\infty}/\fkm_{\psi} \cdot E^{\infty}\simeq \BC,$$ 
 where $\fkm_{\psi}$ is the maximal ideal in $\CS(\BR^{\times})$ consisting of functions that is zero at $x=1$. On the other hand, 
$\bigcap_k \ov{\fkv_2^k E^{\infty}}=E^{\infty}$ since $\fkv_2^k E^{\infty}\supset \SInd_{V_2}^{M_2}(\psi)$ for any integer $k$. Consequently, the exact sequence in theorem~\ref{Canonical filtration} fails in this case. This example also shows that for a fixed open $L$-orbit $\CO\subset\widehat{N}$, $\CS(\CO)\cdot \pi$ is not closed for arbitrary $\pi\in\Smod_P$.

The failure of the nuclear property is localized in a neighborhood of $0$. For simplicity, we only consider $\BR_{>0}$ and identify the space with 
\[
 \{f\in L^2(\BR,dt)\mid (\frac{d}{dt})^k(e^{mt}f)\in L^2,\forall k, m\in \BZ_{\geq 0}\}
\]
via the exponential map. Around $t\to -\infty$, the topology is determined by semi-norms $p_k(f):=\sum_{i=0}^k|(\frac{d}{dt})^if|_{L^2}$, which is essentially the topology of infinite-order Sobolev space, see Example~\ref{Sobolev example}. Then the claim follows from the lemma.
\begin{lemma}\label{non-nuclear lem}
  Let $E = H_{\infty}(\mathbb{R})$ be the infinite-order Sobolev space on $\mathbb{R}$. Then the completion of $E$ with respect to $p_k$ is $H_k(\mathbb{R})$. Moreover, for any $k < s$, the continuous inclusion $H_s \hookrightarrow H_k$ is not compact, and therefore not nuclear.
\end{lemma}
\begin{proof}
    We construct a uniformly bounded sequence in $H_s$ that does not have a convergent subsequence in $H_k$. Let $f\neq 0\in \CC_c^{\infty}(\BR)$ whose support is on $[-\frac{1}{2},\frac{1}{2}]$. Let $f_n(x)=f(x+n)$. Then $\{f_n\}$ is a sequence we want.
\end{proof}

\section{Category $\HC(\fkg,K_L)$}\label{alg complete cat sec}
In this section, $P$ is a standard parabolic subgroup whose unipotent radical $N$ is \textbf{not assumed to be abelian}. In the following sections, if the parabolic subgroup is clear, we will omit $N$ from the subscript of the Casselman-Jacquet functor for simplicity. We define a category $\HC(\fkg,K_L)$ which is an algebraic analogy of the category $\CC(\fkg,L)_f$ introduced in~\cite{WZ}. We will show that this category is equivalent to the category $\CC(\fkg,L)_f$ through a certain globalization functor. 

We first recall some notation in~\cite{WZ}. For a $(\fkl,K_L)$-module $\tau$, we define the generalized $\fkz_L$-weight subspace of weight $\alpha\in\fkz_L^*$ by
\[
 \tau_{\alpha}:= \{v\in\tau\mid (X-\alpha(X))^kv=0\ \text{ for some } k\in\BZ_{> 0}, \forall X\in \fkz_L\}.
\]
Moreover, let $\mathrm{wt}(\tau)$ denote the set of generalized $\fkz_L$-weights of $\tau$ such that the weight space is non-zero. The set of $\fkz_L$-weights in $\U(\fkn)$ is denoted by $\Omega$. We define a partial order on $\fkz_L^*$ as follows:
\begin{equation}\label{partial_order}
\alpha \leq \kappa \quad \text{ if and only if } \quad \kappa-\alpha\in \Omega.
\end{equation}
\begin{definition}
    A $(\fkg,K_L)$-module is a vector space equipped with compatible $\U(\fkg)$-action and $K_L$-action. Let $\HC(\fkg,K_L)$ be the category of $(\fkg,K_L)$-modules $V$ such that
   \begin{enumerate}
       \item  $V/\fkn V$ is a Harish-Chandra module of $(\fkl,K_L)$.
       \item The $\U(\fkg)$-action is $\fkn$-complete. That is, the canonical map 
       $$V\lra \varprojlim_k \ V/\fkn^kV$$
       is an isomorphism. Morphisms in this category are $\U(\fkg)$ and $K_L$-equivariant linear maps.
   \end{enumerate}
\end{definition}

Assume that $V\in\HC(\fkg,K_L)$. Let $V^{[\fkz_L]}$ be the submodule of $V$ consisting of $\fkz_L$-finite vectors. Then $\ov{\fkn}$-action on $V^{[\fkz_L]}=\oplus _{\kappa\in\fkz_L^*} V_{\kappa}$ is $\ov{\fkn}$-locally finite. For every $k$ and every $\kappa\in\fkz_L^*$, we have a natural projection map $V/\fkn^k V\lra (V/\fkn^k V)_{\kappa}$ since $V/\fkn^k V$ is a Harish-Chandra module. By condition (2) of the definition, this defines a $(\fkl,K_L)$-equivariant projection map
\[
pr_{\kappa}: V\lra V_{\kappa} \text{ for all } \kappa\in\fkz_L^*.
\]
Moreover, the projection map gives the isomorphism
\[
V\simeq \prod_{\kappa\in \fkz_L^*} V_{\kappa} \simeq \varprojlim_{S\subset \fkz_L^*\text{ finite}} \bigoplus_{\kappa\in S} V_{\kappa}.
\]
For a fixed $\kappa\in \fkz_L^*$, we have the decomposition
\[
 V \simeq V_{\kappa} \bigoplus (\prod_{\alpha\neq \kappa} V_{\alpha} ), \quad  v \mapsto ( v_{\kappa}, v_{\neq \kappa}).
\]

We first study the action of $\CZ(\fkg)$ on objects of $\HC(\fkg,K_L)$. Recall that the natural projection of $\CZ(\fkg)$ to $\U(\fkp)$ under 
\begin{equation}\label{decom}
    \U(\fkg)\simeq \U(\fkp)\oplus \U(\fkg)\ov{\fkn} \qquad X\mapsto (X_{\fkp},X_{\ov{\fkn}})
\end{equation}
is injective with the image contained in $\CZ(\fkl)$. We introduce the notion of $\CZ(\fkg)$-finite following~\cite{Wal92}.
\begin{definition}
    A $(\fkg,K_L)$-module $M$ is called $\CZ(\fkg)$-finite if there exists an ideal $I\subset \CZ(\fkg)$ of finite codimension which annihilates $M$.
\end{definition}
We need the following lemma to prove that objects in $\HC(\fkg,K_L)$ are $\CZ(\fkg)$-finite.
\begin{lemma}\label{inf-lem}
    Let $V\in\HC(\fkg,K_L)$, and let $M$ be an irreducible $(\fkl,K_L)$-submodule. Suppose that there exists a positive integer $k$ such $\ov{\fkn}^k M=0$, then there exist finitely many algebra homomorphisms $\alpha_1,\dots,\alpha_s:\CZ(\fkg)\lra \BC$ such that 
    \[
    (z-\alpha_1(z))\cdots (z-\alpha_s(z))m=0, \forall m\in M \text{ and } z\in \CZ(\fkg).
    \]
\end{lemma}
\begin{proof}
    Let $\alpha:\CZ(\fkl)\lra \BC$ be the infinitesimal character of $M$. We also regard it as an algebra homomorphism  
    \[ \CZ(\fkg)\lra \CZ(\fkl)\lra \BC
    \]
    by isomorphism~\ref{decom}. We prove the statement by induction on $k$.
    When $k=1$, we have
    \[
    (z-\alpha(z))m= (z_{\fkp} -\alpha(z_{\fkp}))m=0.
    \]
    Suppose that the statement holds for $k$, we proceed to prove the statement assuming that $\ov{\fkn}^{k+1} M=0$. 

    Since $\ov{\fkn}M$ is a Harish-Chandra module, by induction hypothesis, there exists 
    \[
    P(z)=\prod_{i=1}^s (z-\alpha_i(z))  ,
    \]
    where $\alpha_i$ is an algebra homomorphism $\CZ(\fkg)\to \BC$ such that $P(z) x=0$ for $\forall x\in \ov{\fkn}M$. Then 
 \begin{equation*}
     \begin{split}
         P(z)(z-\alpha(z))m&=P(z)(z_{\fkp}-\alpha(z_{\fkp}))m +\sum_{j=1}^t P(z) X_j Y_j m\\
        &=\sum_{j=1}^t X_j P(z)(Y_j m )=0
     \end{split}
 \end{equation*}
 for $\forall m\in M$, where $Y_j\in \ov{\fkn}$ and $\sum_{j=1}^t X_j Y_j= z_{\ov{\fkn}}$. 
\end{proof}

We list some basic properties of the objects in $\HC(\fkg,K_L)$.
\begin{lemma}\label{basic property}
    Let $V\in\HC(\fkg,K_L)$. Then:
    \begin{enumerate}
        \item $V^{[\fkz_L]}$ is $\U(\fkp)$-finitely generated. And the action of $K_L$ on $V^{[\fkz_L]}$ is locally finite.
        \item $V$ is $\CZ(\fkg)$-finite.
        \item Let $M$ be a submodule of $V$. Then the subobject generated by $M$ is 
\[
\ov{M}:=\varprojlim_{S\subset \fkz_L^*\text{ finit}} \bigoplus_{\kappa\in S} pr_{\kappa}(M).
\]
    \end{enumerate}
\end{lemma}
\begin{proof}
    \begin{enumerate}
        \item Since $V/\fkn V$ is a Harish-Chandra module, there exists a sufficiently large $k$ such that $\wt(V/\fkn V)\cap \wt(\fkn^kV/\fkn^{k+\ell}V)=\emptyset$ for any positive integer $\ell$. In other words, $V_{\kappa}$ is a Harish-Chandra module for any $\kappa\in \fkz_L^*$. Hence, $V^{[\fkz_L]}$ is generated by the finitely $\U(\fkl)$-generated module
        \[
        \bigoplus_{\kappa\in \wt(V/\fkn V)} V_{\kappa}
        \]
       under $\U(\fkn)$-action. Moreover, by the condition (2) in the definition, $V^{[\fkz_L]}=\oplus_{\kappa\in\fkz_L^*} V_{\kappa}$. Therefore, it is $K_L$-locally finite.
       
       \item  To prove that $V$ is $\CZ(\fkg)$-finite, it suffices to prove that $V^{[\fkz_L]}$ is $\CZ(\fkg)$-finite. By (1) of Lemma~\ref{basic property}, we can find a submodule $M$ which is a Harish-Chandra $(\fkl,K_L)$-module and generates $V^{[\fkz_L]}$ under $\U(\fkn)$-action. Moreover, there exists some positive integer $k$ such that $\ov{\fkn}^kM=0$ since $V^{[\fkz_L]}$ is $\ov{\fkn}$-locally finite. Consequently, the result follows from Lemma~\ref{inf-lem}.
       
       \item  It is not hard to see that $\ov{M}$ is a $(\fkg,K_L)$-module. Hence, it suffices to show that $\ov{M}\in\HC(\fkg,K_L)$. We first claim that 
       \[
     \ov{M}_{\alpha}\lra (\ov{M}/\fkn\ov{M})_{\alpha}
       \]
       is surjective for any $\alpha\in \fkz_L^*$. Let $x\in \ov{M}$ such that $\ov{x}$, the image of $x$ in $\ov{M}/\fkn\ov{M}$, has generalized $\fkz_L$-weight $\alpha$. Then there exists some positive integer $k$ such that
       \[
       (z-\alpha(z))^k x_{\neq \alpha}=(z-\alpha(z))^k (x-x_{\alpha})\in \fkn\ov{M} \text{ for } \forall z\in\fkz_L. 
       \]
       Since there exists some $z\in\fkz_L$ such that the action of $(z-\alpha(z))^k$ on $\prod_{\kappa\neq \alpha} pr_{\kappa} (M)$ is invertible, we have
       \[
       x_{\neq \alpha}\in \fkn\ov{M}.
       \]
       Thus, the claim follows.
       
       Note that $\ov{M}$ is $\CZ(\fkg)$-finite by (2). Hence, $\ov{M}/\fkn\ov{M}$ is $\CZ(\fkl)$-finite. In particular, it is $\fkz_L$-finite. In other words, there exist finitely many weights $\kappa_1,\dots,\kappa_s$, such that 
       \[
       \ov{M}/\fkn\ov{M}= \bigoplus_{i=1}^s (\ov{M}/\fkn\ov{M})_{\kappa_i}.
       \]
       Since $V_{\kappa_i}$ is Harish-Chandra for each $i$, $\ov{M}/\fkn\ov{M}$ is also Harish-Chandra. Now, we can describe the natural map
       \[
       \ov{M}\lra \varprojlim_{k} \ov{M}/\fkn^k\ov{M}
       \]
       more explicitly. It is given by the projective limit of
       \[
       \ov{M} \stackrel{pr}{\lra} \bigoplus_{\kappa\in \wt(\ov{M}/\fkn^k \ov{M})} \ov{M}_{\kappa} \lra \ov{M}/\fkn^k \ov{M},
       \]
       where the last map is the quotient map. Moreover, we can construct an inverse. Let $\kappa\in \fkz_L^*$. Since $V_{\kappa}\simeq (V/\fkn^k V)_{\kappa}$ for a sufficiently large $k$, $(\ov{M}/\fkn^k\ov{M})_{\kappa}\simeq \ov{M}_{\kappa}$ by a similar argument as the claim. Then the inverse is given by the projective limit of
       \[
       \ov{M}/\fkn^k \ov{M}\lra (\ov{M}/\fkn^k \ov{M})_{\kappa}\simeq \ov{M}_{\kappa}.
       \]  
    \end{enumerate}
\end{proof}

We give an equivalent characterization of the category $\HC(\fkg,K_L)$. We observe that objects in the category $\HC(\fkg,K_L)$ admit an action of a partial completion of $\U(\fkg)$. Define
\[
 \CE_{\fkn}(\fkg):= \varprojlim_k \U(\fkg)/\fkn^k \U(\fkg) .
\]
When $\fkn$ is clear from the context, we omit this subscript.  Let $a=(a_k), b=(b_k)\in \CE(\fkg)$. We define an element $ab \in \CE(\fkg)$ as follows. For any positive integer $k$, we fix a pre-image $X$ of $a_k$ in $\U(\fkg)$. Then there exists a sufficiently large positive integer $k'$ such that $X\fkn^{k'}\subset \fkn^k \U(\fkg)$. Let $Y$ be a pre-image of $b_{k'}$ in $\U(\fkg)$. Then
\[
 (ab)_k := p_k(  XY), 
\]
 where $p_k$ is the canonical projection $\U(\fkg) \to \U(\fkg)/\fkn^k \U(\fkg)$. It is not hard to see that the definition is independent of the choice of $X,k',Y$. Moreover, it defines an algebra structure on $\CE_(\fkg)$, such that the canonical embedding $\U(\fkg) \hookrightarrow \CE(\fkg)$
 is an algebra homomorphism. $\CE(\fkg)$ contains a subalgebra
 \[
 \CE(\fkp)=\CE_{\fkn}(\fkp):= \varprojlim_k \U(\fkp)/\fkn^k \U(\fkp) .
 \]
 Note that $\{\fkn^k \U(\fkp)= (\fkn\U(\fkp))^k\}$ is a decreasing filtration of two-sided ideals in $\U(\fkp)$. Hence, by passing to associated graded algebra, $\CE(\fkp)$ is a Noetherian ring by the fact that $\U(\fkp)$ is Noetherian, see~\cite[Theorem 10.26]{AM18}. 

Let $V$ be a $\U(\fkg)$-module. Then $\CE(\fkg)$ acts on $\varprojlim_k V/\fkn^k V$ via the projective limit of the $\U(\fkg)/\fkn^k \U(\fkg)$-actions on $V/\fkn^k V$. Note that $K_L$ acts on $\CE(\fkg)$ by the adjoint action. Hence, by a $(\CE(\fkg), K_L)$-module we mean a vector space equipped with compatible actions of $\CE(\fkg)$ and $K_L$ that is $K_L$-locally finite.

 \begin{proposition}\label{equivalent char}
     \begin{enumerate}
     \item If $V\in\HC(\fkg,K_L)$, then $V$ is finitely generated over $\CE(\fkp)$.
     \item If $V$ is a finitely generated $\CE(\fkp)$-module, then $V\simeq \varprojlim_k V/\fkn^k V$.
 \end{enumerate} 
 \end{proposition}
\begin{proof}
(1) follows by an argument similar to that of (1) Lemma~\ref{basic property}. We leave the proof for readers. We prove (2). Since $\CE(\fkp)$ is a Noetherian ring, we have a finite presentation of $V$
\[
\begin{tikzcd}
    \bigoplus_{i=1}^p \CE(\fkp) \ar[r]\ar[d] & \bigoplus_{j=1}^q \CE(\fkp) \ar[r]\ar[d]& V\ar[r]\ar[d] &0\\
     \bigoplus_{i=1}^p \CE(\fkp) \ar[r] &\bigoplus_{j=1}^q \CE(\fkp) \ar[r]& \varprojlim_k V/\fkn^k V \ar[r] &0.
\end{tikzcd}
\]
In the diagram, we use the fact that $\CE(\fkp)\simeq \varprojlim_k \CE(\fkp)/\fkn^k \CE(\fkp)$. By the Artin-Rees lemma, the bottom horizontal complex is exact, which implies the result.

\end{proof}

\begin{corollary}\label{abelian cate alg}
     The category $\HC(\fkg,K_L)$ is an abelian category, where kernels and cokernels are the same as those in the ambient category of $(\fkg,K_L)$-modules.
\end{corollary}

\begin{proof}
   
     Suppose that there is an exact sequence of $(\fkg,K_L)$-modules
    \[
    0 \lra M \lra V \lra N 
    \]
    such that $N,V\in\HC(\fkg,K_L)$. We observe that
      \[
      0\lra M_{\kappa}\lra V_{\kappa}\lra N_{\kappa}
      \]
      is exact for any $\kappa\in\fkz_L^*$. Thus, for any $\kappa\in \fkz_L^*$, $pr_{\kappa}(M)= M_{\kappa}$. 
      Hence, we have $M=\ov{M}$, which implies the result by Lemma~\ref{basic property} (3).  

       The fact that the cokernel exists in this category directly follows from Proposition~\ref{equivalent char}.
\end{proof}
\begin{example}
    When $P$ is a minimal parabolic subgroup of $G$, then there is a functor
    \[
    \HC(\fkg,K_L) \lra \CO^{\fkp} \quad V \mapsto V^{[\fkz_L]},
    \]
    where $\CO^{\fkp}$ is the parabolic category $\CO$, and $ V^{[\fkz_L]}$ forgets the $K_L$-action.
\end{example}

\begin{proposition}
   Suppose that there is a short exact sequence of $(\CE(\fkg),K_L)$-modules 
   \[
   0\lra M\lra V\lra N\lra 0
   \]
   such that $M,N\in\HC(\fkg,K_L)$. Then $V\in\HC(\fkg,K_L)$.
\end{proposition}
\begin{proof}
  By the right exactness of the Jacquet functor, $V/\fkn V$ is a Harish-Chandra module. Then the result follows from Proposition~\ref{equivalent char}.
\end{proof}

\subsection{Classification of irreducible objects}\label{classify irr subsec}
Similar to category $\CC(\fkg,L)$, we first define a family of standard objects in category $\HC(\fkg,K_L)$.
\begin{definition}
    Given a Harish-Chandra module $\tau$ of $(\fkl,K_L)$, the \textbf{formal Verma module} $\CV(\tau)$ is defined as the inverse limit
    \[
    \varprojlim_{k\geq 0}( \U(\fkg)\otimes_{\U(\ov{\fkp})}\tau) /(\fkn^k\U(\fkg)\otimes_{\U(\ov{\fkp})}\tau),
    \]
    where $\tau$ extends trivially to be a $\U(\ov{\fkp})$-module.
\end{definition}

It is easy to check that $\CV(\tau)\in\HC(\fkg,K_L)$ and the functor ``$\tau\mapsto \CV(\tau)$" is an exact functor from $\HC(\fkl,K_L)$ to $\HC(\fkg,K_L)$. 

\begin{lemma}
  Let $\tau$ be an irreducible Harish-Chandra module of $(\fkl,K_L)$. Then the formal Verma module $\CV(\tau)$ has a unique maximal proper subobject, and hence a unique non-zero irreducible quotient object.
\end{lemma}
\begin{proof}
    Note that $\CV(\tau)$ is generated as an object in $\HC(\fkg,K_L)$ by $\tau$ under $\U(\fkn)$-action. Hence, for any proper subobject $M_1,M_2$, we have $pr_{\kappa}(M_1)=pr_{\kappa}(M_2)=0$ when $\kappa=\wt(\tau)$. This implies that $pr_{\kappa}(\ov{M_1+M_2})=0$. In other words, the sum of all proper subobjects is the unique maximal proper subobject.
\end{proof}

We denote the unique irreducible quotient of $\CV(\tau)$ by $\CL(\tau)$. Note that $\CL(\tau_1)\simeq \CL(\tau_2)$ if and only if $\tau_1\simeq \tau_2$ as $L$-representation. On the other hand, we have the following lemma, the proof of which is the same as \cite[Lemma 2.22]{WZ}:
\begin{lemma}
    Let $V$ be an irreducible object in $\HC(\fkg,K_L)$. Then there exists an irreducible Harish-Chandra module $\tau$ of $(\fkl,K_L)$, such that $V$ is a quotient of the formal Verma module $\CV({\tau})$.
\end{lemma}

Hence, we get the classification of irreducible objects in $\HC(\fkg,K_L)$:
\begin{corollary}
    There is a one-to-one correspondence between irreducible Harish-Chandra $(\fkl,K_L)$-modules and irreducible objects in $\HC(\fkg,K_L)$, given by the map
    \[
    \tau \mapsto \CL(\tau).
    \]
\end{corollary}
Moreover, if the infinitesimal character of $\tau$ is $\chi_{\lambda}$, then the infinitesimal character of $\CV(\tau)$ is $\chi_{\overline{\lambda-\rho+\rho_{\fkl}}}$, where $\overline{\lambda}$ is the image of $\lambda$ under the following natural projection:
\begin{equation*}
    \fka^*  \dslash W_L \lra \fka^* \dslash W .
\end{equation*}
For an infinitesimal character $\chi_{\mu}$ of $\fkg$, we use $\CT_{\mu}$ to denote the set of irreducible Harish-Chandra $(\fkl,K_L)$-modules $\tau$, for which $\CV(\tau)$ has infinitesimal character $\chi_{\mu}$. Then $\CT_{\mu}$ is a finite set.

\begin{proposition}
    The length of the objects in $\HC(\fkg,K_L)$ is finite.
\end{proposition}
\begin{proof}
    Let $V\in\HC(\fkg,K_L)$. Since $V$ is $\CZ(\fkg)$-finite, there are finitely many irreducible $(\fkl,K_L)$-modules $\tau$ such that $\CL(\tau)$ appears in irreducible components of $V$. On the other hand, since $V_{\alpha}$ is Harish-Chandra for any $\alpha\in\fkz_L^*$, each $\CL(\tau)$ appears only finite many times. Consequently, the length of $V$ is finite.
\end{proof}

We fix the datum $(L_{\BC},\iota)$ such that $L_{\BC}$ is a complex connected reductive Lie group and $\iota: L\lra L_{\BC}$ is a Lie group homomorphism satisfying
\begin{itemize}
    \item The Kernel of $d\iota$ is contained in the center of $\Lie(L)$;
    \item the image of $d\iota$ is a real form of $\Lie(L_{\BC})$.
\end{itemize}
For example, $L_{\BC}=\Ad(\fkl)$ with $\iota$ the adjoint representation is one of the choices. Let $Q$ be the lattice of algebraically integral weights of $\fkg$ which are analytic integral with respect to $L_{\BC}$. We identify $Q$ as a lattice in $\fka^*$. Let $\Rep(\fkg,Q)$ be the category of finite dimensional $\fkg$-representation whose weights are contained in $Q$. Then it is the same as the category of finite dimensional $\fkg$-representation which can be integrated as a holomorphic $L_{\BC}$-representation. We also regard it as a $(\fkg,L)$-module through $\iota$.

In future work, we will use coherent continuation to compute the Casselman-Jacquet functor. The following lemma provides the initial idea.
\begin{lemma}
    Let $V$ be a $(\fkg,K_L)$-module and $\beta\in \Rep(\fkg,Q)$. Then we have a natural isomorphism
    \[
    \mathrm{CJ}(V\otimes \beta)\simeq \mathrm{CJ}(V)\otimes \beta.
    \]
    In particular, when $V\in \HC(\fkg,K_L)$, it shows that $V\otimes \beta\in\HC(\fkg,K_L)$.
\end{lemma}
\begin{proof}
  Let 
  \[
  \beta=\beta_0\supset \beta_1\supset \dots\supset \beta_r=0
  \]
  be a finite $(\fkp,L)$-filtration such that $\fkn$-action on $\beta_{i}/\beta_{i+1}$ is trivial. For any fixed $k$, it is not hard to see 
  \[
  \fkn^{k+r} (V\otimes \beta)\subset \fkn^k V\otimes \beta,
  \]
  and 
  \[
 \fkn^{k+r} V\otimes \beta \subset \fkn^k(V\otimes \beta).
  \]
  Consequently,  
  \[
  \varprojlim_k (V\otimes \beta)/\fkn^k(V\otimes \beta)\simeq (\varprojlim_k V/\fkn^k V)\otimes \beta. \]
  Moreover, when $V\in \HC(\fkg,K_L)$, we observe that $(V\otimes \beta)/\fkn(V\otimes\beta)$ is Harish-Chandra since
  \[
 ( V\otimes (\beta_i/\beta_{i+1}))/\fkn ( V\otimes (\beta_i/\beta_{i+1}))\simeq (V/\fkn V )\otimes (\beta_i/\beta_{i+1})
  \]
  is Harish-Chandra for any integer $i$.
\end{proof}

\subsection{A commutative diagram}
It is not hard to see that $\mathrm{CJ}(\pi)\in\HC(\fkg,K_L)$ when $\pi\in \HC(\fkg,K)$. Hence, the Casselman-Jacquet functor is a functor
\[
\mathrm{CJ}: \HC(\fkg,K)\lra \HC(\fkg,K_L),
\]
which is exact by the Artin-Rees lemma in~\cite{CO78}. On the other hand, following lemma shows that $\mathrm{CJ}^{\infty}$ is a functor
\[
\mathrm{CJ}^{\infty}: \mathrm{CW}_G\lra  \CC(\fkg,L)_f,
\]
which is also exact by~\cite[Theorem 12.4]{CWYZ}.
\begin{lemma}\label{CW CJ in finite length cat}
    Let $\pi$ be a Casselman-Wallach representation of $G$. Then $\mathrm{CJ}^{\infty}(\pi)\in\CC(\fkg,L)_f$.
\end{lemma}
\begin{proof}
    By~\cite[Lemma 2.24]{WZ}, it suffices to show that $\mathrm{CJ}^{\infty}(\pi)\in\CC(\fkg,L)$. For simplicity, let $V:=\mathrm{CJ}^{\infty}(\pi)$. On the one hand, 
    \[
    V^{[\fkz_L]}=\bigoplus_{\alpha\in\fkz_L^*} V_{\alpha}.
    \]
    On the other hand, by~\cite[Theorem 12.1]{CWYZ}, we know that \[
    \wt(V)\subset \wt(\pi/\ov{\fkn\pi})+\Omega.
    \]
    Hence, $V^{[\fkz_L]}$ is $\ov{\fkn}$-locally finite. Moreover, for any $\alpha\in\fkz_L^*$, there exists a positive integer $s$ such that 
    \[
    V_{\alpha}= (\pi/\ov{\fkn^s\pi})_{\alpha},
    \]
    which is Casselman-Wallach. Moreover, since two directed sets 
    \[
    \{ \ov{\bigoplus_{\alpha\in\fkz_L^*\setminus S} V_{\alpha}}\mid S\subset \fkz_L^* \text{ finite}\} \text{ and } \{\ov{\fkn^k V}\}
    \]
    are co-finial, we have
    \[
    V\simeq \varprojlim_{k} V/\ov{\fkn^k V}\simeq \mathop{\varprojlim}\limits_{S\subset \fkz_L^* \text{ finite}}V/\overline{ \mathop{\oplus}\limits_{\alpha\in \fkz_L^*\setminus S}V_{\alpha}}.
    \]
\end{proof}

Now, we want to define a functor $\widehat{\cdot}$ such that the following diagram is commutative.
    \begin{equation}\label{func-comm-dia}
        \begin{tikzcd}
    \HC(\fkg,K)\ar[r,"\cdot^{\infty}"]\ar[d,"\mathrm{CJ}"] & \mathrm{CW}_G \ar[d,"\mathrm{CJ}^{\infty}"] \\
             \HC(\fkg,K_L)\ar[r,"\widehat{\cdot}"] & \CC(\fkg,L)_f .
        \end{tikzcd}
    \end{equation}
    We first give a definition based on the Casselman-Wallach globalization of Harish-Chandra $(\fkl,K_L)$-module.
\begin{definition}
 Let $V\in\HC(\fkg,K_L)$. We define the $(\fkg,L)$-module $\widehat{V}$ as 
    \[
    \widehat{V}:=\varprojlim_k (V/\fkn^k V)^{\infty}.
    \]
    equipped with the inverse limit topology.
\end{definition}
We specify the action of the Lie algebra. It is similar as~\cite[Lemma 9.1]{CWYZ}. 
\begin{itemize}
    \item The $\fkn$-action: For every $k$, there is a $L$-equivariant continuous map
    \[
   \mu_k:  \fkn \otimes (V/\fkn^k V)^{\infty} \lra (V/\fkn^{k+1} V)^{\infty},
    \]
    which is the Casselman-Wallach globalization of the action map $\fkn \otimes (V/\fkn^k V) \to (V/\fkn^{k+1} V)$, since $\left(\fkn \otimes (V/\fkn^k V)\right)^{\infty}=\fkn \otimes (V/\fkn^k V)^{\infty}$. The $\fkn$-action is defined as 
    \[
    X\cdot (x_1,x_2,\cdots):= (0,\mu_1(X\otimes x_1),\mu_2(X\otimes x_2),\cdots)
    \]
     for $X\in \fkn$ and $(x_1,x_2,\cdots)\in\widehat{V}$.
    \item The $\ov{\fkn}$-action: For every $k$, there is a sufficiently large $n_k$, such that $\ov{\fkn} \fkn^{n_k}\subset \fkn^k\U(\fkg)$. Thus, by Casselman-Wallach globalization, there is a $L$-equivariant continuous map
    \[
   \lambda_k:  \ov{\fkn} \otimes (V/\fkn^{n_k} V)^{\infty} \lra (V/\fkn^{k} V)^{\infty}.
    \]
    The $\ov{\fkn}$-action is defined as 
    \[
    Y\cdot (x_1,x_2,\cdots):= (0,\lambda_1(X\otimes x_{n_1}),\lambda_2(X\otimes x_{n_2}),\cdots)
    \]
     for $Y\in \ov{\fkn}$ and $(x_1,x_2,\cdots)\in\widehat{V}$. The $\fkn$ and $\ov{\fkn}$-actions glue to be a Lie algebra action, namely, $[X,Y]\cdot x= X\cdot Y\cdot x- Y\cdot X\cdot x$ for any $X\in\fkn,Y\in\ov{\fkn}$ and $x\in\widehat{V}$ by the Casselman-Wallach globalization of the following commutative diagram
     \[
     \begin{tikzcd}
         \fkn \otimes \ov{\fkn} \otimes V/\fkn^N V \ar[r,"\alpha"]\ar[d,"\beta"] & V/\fkn^k V \ar[d,"id"]\\
      ( \fkn \oplus \ov{\fkn} \oplus \fkl)  \otimes V/\fkn^N V \ar[r,"\gamma"] &  V/\fkn^k V ,
     \end{tikzcd}
     \]
     where $N$ is a sufficiently large integer, 
     $$\alpha(X\otimes Y\otimes v):= \mu_{k-1} (X\otimes \lambda_{k-1}(Y\otimes v))- \lambda_k( Y\otimes \mu_{N}(X\otimes v)),$$
     $\beta(X\otimes Y\otimes v):=[X,Y]\otimes v$, and $\gamma$ is the direct sum of action maps. All of the maps involved are $(\fkl,K_L)$-equivariant.
\end{itemize}
By comparing the generalized weight space, we obtain
     \begin{equation}\label{glob-iso}
           \widehat{V}\simeq \varprojlim_{S\subset \fkz_L^* \text{ finite}} \bigoplus_{\kappa\in S} (V_{\kappa})^{\infty} \simeq \prod_{\kappa\in \fkz_L^*} (V_{\kappa})^{\infty}
     \end{equation}
    In particular, the functor $\widehat{\cdot}$ is exact. We show that this $(\fkg,L)$-module is in the category $\CC(\fkg,L)_f$.
 \begin{lemma}
     Let $V\in\HC(\fkg,K_L)$. Then $\widehat{V}\in\CC(\fkg,L)_f$.
 \end{lemma}
 \begin{proof}
    By isomorphism~\eqref{glob-iso}, we know that $V$ is dense in $\widehat{V}$. Hence $\widehat{V}$ is $\CZ(\fkg)$-finite. By \cite[Lemma 2.24]{WZ}, it suffices to show that $\widehat{V}\in\CC(\fkg,L) $. Then it follows from the fact that $\widehat{V}_{\kappa}=(V_{\kappa})^{\infty}$.
 \end{proof}

 Now, we can prove that the diagram~\eqref{func-comm-dia} is commutative.
\begin{proposition}\label{commutative dia prop}
    The diagram~\eqref{func-comm-dia} is commutative. That is, there exists an isomorphism of functors
    \[
    \widehat{\cdot}\circ \mathrm{CJ}\simeq \mathrm{CJ}^{\infty}\circ \cdot ^{\infty}.
    \]
\end{proposition}
\begin{proof}
    Let $\pi\in \HC(\fkg,K)$. Then 
    \[
 \widehat{   \mathrm{CJ}(\pi)} = \varprojlim_k (\pi/\fkn^k \pi)^{\infty} 
    \]
    and 
    \[
    \mathrm{CJ}^{\infty}(\pi^{\infty}) = \varprojlim_k \pi^{\infty}/\ov{\fkn^k \pi^{\infty}}.
    \]
    Then, the natural isomorphism is given by~\cite[Theorem 12.1]{CWYZ}.
\end{proof}

\begin{example}\label{Verma exa}
    The functor sends formal Verma modules to formla Verma modules. Let $\tau$ be a Harish-Chandra module of $(\fkl,K_L)$, then
    \[
    \widehat{\CV(\tau)}\simeq \CV(\tau^{\infty}).
    \]
\end{example}

\subsection{Category equivalence}
In this subsection, we prove that the functor $\widehat{\cdot}$ defined in the last subsection is a category equivalence. 

Let $V\in\CC(\fkg,L)_f$. By the argument in \cite[Proposition 2.26, Remark 2.31]{WZ}, we know that $V/\fkn^k V$ is a Casselman-Wallach representation of $L$. In particular, $\fkn^k V$ is closed in $V$. In order to prove the category equivalence, we need the following lemma.
 \begin{lemma}\label{inv lim lem}
     If $V\in \CC(\fkg,L)_f$, then 
     \[
     V\simeq \varprojlim_k  V/ \fkn^k V.
     \]
 \end{lemma}
 \begin{proof}
     On the one hand, we have the short exact sequence
     \[
     0\lra (\fkn^k V)_{\kappa} \lra V_{\kappa}\lra (V/\fkn^k V)_{\kappa}
     \]
     for any $\kappa\in\fkz_L^*$. On the other hand, since $V^{[\fkz_L]}$ is dense in $V$, $(\fkn^k V)^{[\fkz_L]}$ is also dense in $\fkn^k V$ for any $k$. Hence, two directed sets
      \[
    \{ \ov{\bigoplus_{\kappa\in\fkz_L^*\setminus S} V_{\kappa}}\mid S\subset \fkz_L^* \text{ finite}\} \text{ and } \{\fkn^k V\}
    \]
    are co-finial, which implies the result.
 \end{proof}

 Now, we define the quasi-inverse of the functor $\widehat{\cdot}$. 
 \begin{definition}
     Let $V\in\CC(\fkg,L)_f$. We define the $(\fkg,K_L)$-module $V^{\flat}$ as
     \[
     V^{\flat}:=\varprojlim_k (V/\fkn^k V)^{[K_L]}.
     \]
 \end{definition}
The Lie algebra action is defined similarly to the globalization functor $\widehat{\cdot}$. Details are left for interested readers.
 \begin{lemma}\label{equi-lem}
      Let $
V\in \CC(\fkg,L)_f$. Then $V^{\flat}/\fkn^k V^{\flat}\simeq (V/\fkn^k V)^{[K_L]}$. In particular, $V^{\flat}\in\HC(\fkg,K_L)$.

 \end{lemma}
\begin{proof}
Since $(V/\fkn^p V)^{[K_L]} \to (V/\fkn^q V)^{[K_L]}$ is surjective for any $p>q$, the natural map
\[V^{\flat}/\fkn^k V^{\flat}\lra (V/\fkn^k V)^{[K_L]} \quad (x_j)_{j=1}^{\infty} \mapsto x_k
\]
is also surjective. Let $(x_j)\in V^{\flat}$ such that $x_j\in (\fkn^kV/\fkn^j V)^{[K_L]}$ for any integer $j$. It suffices to prove that $(x_j)\in \fkn^k V^{\flat}$. Suppose that in the following $L$-equivariant commutative diagram
    \begin{center}
        \begin{tikzcd}
            \fkn\otimes (V/\fkn^j V)^{[K_L]}  \ar[d," a_j"] &  \fkn\otimes (V/\fkn^{j+1} V)^{[K_L]} \ar[d," a_{j+1}"]\ar[l,"\widetilde{d_j}=\id\otimes d_j"]\\
             (\fkn V/\fkn^j V)^{[K_L]}  & (\fkn V/\fkn^{j+1} V)^{[K_L]}\ar[l,"d_j"],
        \end{tikzcd}
    \end{center}
    we have $x_j\in   (\fkn V/\fkn^j V)^{[K_L]}, x_{j+1}\in (\fkn V/\fkn^{j+1} V)^{[K_L]}$ and 
    \[ y_j  =\sum_{i=1}^N X_i\otimes v_i \in \fkn\otimes (V/\fkn^j V)^{[K_L]}
    \]
    such that $a_j(y_j)=x_j$ and $d_j(x_{j+1})=x_j$. Here, ``$d$" is the natural quotient map, and ``$a$" is the $\fkn$-action map. Then, the subset
    \[
     W:=   \sum_{i=1}^N X_i\otimes (v_i+(\fkn^{j-1}V/\fkn^j V)^{[K_L]})\subset a_j^{-1}(x_j)
    \]
    satisfies $ \widetilde{d_{j-1}}(W)= \widetilde{d_{j-1}}(y_j)$ and 
    \[
  x_{j+1}\in  a_{j+1}\left( \widetilde{d_j}^{-1}(W)\right) = d_j^{-1}(x_j).
    \]
    Therefore, we can successively define $(y_j)$ such that $(y_j)\in \fkn V^{\flat}$. This completes the proof. 
\end{proof}
We observe that similar to~\eqref{glob-iso}, by Lemma~\ref{inv lim lem}, 
\begin{equation}\label{prod dec}
     V^{\flat}\simeq  \prod_{\kappa \in\fkz_L^*} (V_{\kappa})^{[K_L]}.
\end{equation}
In particular, the functor $\cdot ^{\flat}$ is exact. Now, we can prove the category equivalence.
\begin{theorem}\label{category equi thm}
    The functor $\widehat{\cdot}$ is an equivalence of categories, whose quasi-inverse is $\cdot ^{\flat}$. 
\end{theorem}
\begin{proof}
    The theorem follows from Lemma~\ref{equi-lem}, isomorphism~\eqref{prod dec},\eqref{glob-iso} and the fact that Casselman-Wallach globalization and taking $K_L$-finite vectors are quasi-inverse to each other.
\end{proof}
The theorem has some direct corollaries.
\begin{corollary}
    Let $\tau$ be an irreducible $(\fkl,K_L)$-module. Then 
    \[
   \widehat{ \CL(\tau)}\simeq \CL(\tau^{\infty}).
    \]
\end{corollary}
\begin{proof}
    Since $\CL(\tau)$ and $\CL(\tau^{\infty})$ are characterized as the unique irreducible quotient objects of $\CV(\tau)$ and $\CV(\tau^{\infty})$, the statement follows from example~\ref{Verma exa}.
\end{proof}

\begin{corollary}\label{finite term coro}
     Let $\chi_{\mu}$ be an infinitesimal character of $\fkg$. If $\tau$ is an element in $\CT_{\mu}$ such that $\mathrm{wt}(\tau)$ is maximal in $\mathrm{wt}(\CT_{\mu})$, then $\CV(\tau)$ is irreducible.
\end{corollary}
\begin{proof}
    It is a direct consequence of Theorem~\ref{category equi thm} and~\cite[Lemma 2.25]{WZ}.
\end{proof}

\begin{corollary}\label{homo HC coro}
    Let $E\in \HC(\fkg,K_L)$. Then for any integer $i$, $\rmh_i(\fkn,E)$ is a Harish-Chandra module of $L$.
\end{corollary}
\begin{proof}
    By Corollary~\ref{finite term coro} and the classification of irreducible objects, we can reduce the problem to formal Verma modules. For details, see step 2 of~\cite[Proposition 2.26]{WZ}.
\end{proof}

Actually, the result in subsection~\ref{classify irr subsec} can also be obtained from the category equivalence and the classification of irreducible objects in $\CC(\fkg,L)_f$.

\begin{corollary}
    The morphism in category $\CC(\fkg,L)_f$ has a closed image. Moreover, $\CC(\fkg,L)_f$ is an abelian category, and its kernels and cokernels agree with those in the ambient category of $(\fkg,L)$-modules. 
\end{corollary}
\begin{proof}
    Let $\alpha:A\to B$ be a morphism in $\CC(\fkg,L)_f$. Let $C:= \Coker(A^{\flat}\to B^{\flat})$. Since the globalization functor is exact, we have an exact sequence by Corollary~\ref{abelian cate alg} and Theorem~\ref{category equi thm},
    \[
     A \lra B \lra \widehat{C}.
    \]
    Thus $\mathrm{Im}(\alpha)$ is closed in $B$. The second assertion follows from a similar argument.
\end{proof}

\section{Homological comparison conjecture}
In this section, we outline a proof of the comparison conjecture~\ref{comparison conj} under the assumption $\spadesuit$. The proof proceeds as follows.
\begin{enumerate}
    \item [Step 1.] We first prove that under the assumption $\spadesuit$, we have
\begin{equation}\label{CW-comp}
    \rmh_i(\fkn,\pi)\simeq \rmh_i(\fkn,\mathrm{CJ}^{\infty}(\pi))
\end{equation}
for any Casselman-Wallach representation $\pi$ and integer $i$.
\item [Step 2.]When $\pi$ is a Harish-Chandra $(\fkg,K)$-module, we prove that
\[
\rmh_i(\fkn,\pi)\simeq \rmh_i (\fkn,\mathrm{CJ}(\pi)).
\]
In fact, this was already mentioned in~\cite{HS83a}.
\item[Step 3.] Then we prove the homological comparison between $\HC(\fkg,K_L)$ and $\CC(\fkg,L)_f$. In particular, we prove that when $E\in \HC(\fkg,K_L)$, we have
\begin{equation}\label{com-in-CJ eq}
  \rmh_i(\fkn,E)^{\infty}\simeq \rmh_i(\fkn,\widehat{E}),  
\end{equation}
where $\rmh_i(\fkn,E)^{\infty}$ is the Casselman-Wallach globalization of $\rmh_i(\fkn,E)\in \HC(\fkl,K_L)$, see Corollary~\ref{homo HC coro}. 
\end{enumerate}

Before the detailed proof, we record a lemma that is crucial in \textbf{Step 1} to show the vanishing of the homology of the non-zero spectrum. Here, the notation is the same as in subsection~\ref{general theory section} and Example~\ref{diff-mod-exam}. 
\begin{lemma}[\cite{WZ}, Corollary 7.6]\label{vanish lem}
    Let $\sigma\in \Smod_{S_{\ell}\cap L}$. Then for any $\ell\neq 0$, we have
    \[
    \rmh_i (\fkn, I_{\ell}(\sigma))=0 \text{ for any integer } i.
    \]
\end{lemma}
\subsection{Details of the proof}
Let $J$ be a generalized principal series of $G$ and assume $\spadesuit$. Then by Theorem~\ref{Canonical filtration}, we have a short exact sequence
\[
0\lra \CS(\widehat{N}\setminus\{0\})\cdot J \lra J\lra \mathrm{CJ}^{\infty}(J)\lra 0,
\]
such that $\rmh_i(\fkn, \CS(\widehat{N}\setminus\{0\})\cdot J)=0$ for any integer $i$ by Lemma~\ref{generalized prin lem}, Lemma~\ref{vanish lem} and~\cite[Lemma 2.9]{WZ}. Therefore, for any integer $i$,
\[
    \rmh_i(\fkn,J)\simeq \rmh_i(\fkn,\mathrm{CJ}^{\infty}(J)).
\]
 Let $\pi$ be a Casselman-Wallach representation. By Casselman's embedding theorem, there is a resolution of $\pi$ by generalized principal series:
    \[
    \pi \lra J_0\lra J_1\lra \dots
    \]
    Considering the $\fkn$-Koszul resolution $P_{i,\bullet}$ of each $J_i$, we obtain a double complex $P_{\bullet,\bullet}$. A standard homological argument shows that
    \[
    \rmh_i(\Tot(P_{\bullet,\bullet}))\simeq \rmh_i(\fkn,\pi).
    \]
     On the other hand, we have a decreasing filtration $\CF^{\bullet}$ of the total complex 
    \[
    \CF^j=F^j(\Tot(P_{\bullet,\bullet})):= \Tot(P_{\geq j,\bullet}).
    \]
    Since the Koszul resolution is of finite length, we can find a large enough $m$, such that for every $i$,
    \[
     \rmh_i(\CF^0/\CF^m)\simeq \rmh_i(\fkn,\pi).
    \]
    
 On the other hand, by the exactness of Casselman-Jacquet functor, we have a corresponding resolution
\[
0\lra \mathrm{CJ}^{\infty}(\pi)\lra \mathrm{CJ}^{\infty}(J_0)\lra \dots,
\]
 double complex $\widetilde{P}_{\bullet,\bullet}$ by $\fkn$-Koszul resolution of $\mathrm{CJ}^{\infty}(J_{\bullet})$ and decreasing filtration $\widetilde{\CF}^{\bullet}$. Note that there is a canonical map $J_p\otimes \fkn^q \lra \mathrm{CJ}^{\infty}(J_p)\otimes \fkn^q$ for any pair of integers $p,q$ intertwining two double complexes, and
\[
\rmh_i(\CF^j/\CF^{j+1}) \simeq \rmh_i(\fkn, J_j)\simeq \rmh_i(\fkn,\mathrm{CJ}^{\infty}(J_j))\simeq \rmh_i(\widetilde{\CF}^{j}/\widetilde{\CF}^{j+1}).
\]

Thus, inductively, we consider the exact sequence of complexes:
    \begin{equation*}
        0\lra \CF^j/\CF^{j+1}\lra \CF^{j-r+1}/\CF^{j+1}\lra\CF^{j-r+1}/\CF^{j}\lra 0 ,
    \end{equation*} 
    which leads to 
    \begin{equation}\label{CW-com-prove}
         \rmh_i(\fkn,\pi)\simeq \rmh_i(\CF^0/\CF^m)\simeq \rmh_i(\widetilde{\CF}^{0}/\widetilde{\CF}^{m})\simeq \rmh_i(\fkn,\mathrm{CJ}^{\infty}(\pi)).   
    \end{equation}
This completes the proof of \textbf{Step 1}.

\begin{proposition}\label{HC-com}
    Let $\pi$ be a Harish-Chandra $(\fkg,K)$-module, then the natural map
    \[
\rmh_i(\fkn,\pi)\lra  \rmh_i (\fkn,\mathrm{CJ}(\pi))
\]
is an isomorphism.
\end{proposition}
\begin{proof}
    By Osborne's lemma, $\pi$ is finitely generated by $\U(\fkn_{B})$. Hence, there exists a resolution $P^{\bullet}$ of $\pi$ by finite rank free $\U(\fkn_{B})$-modules. Note that
    \begin{equation*}
        \rmh_i(\fkn,\U(\fkn_{B}))\simeq\begin{cases}
            \U(\fkn_L) & i=0\\
            0 & i>0.
        \end{cases}
    \end{equation*}
    In particular, this resolution is an acyclic resolution. On the other hand, since the $\mathrm{CJ}$ is exact, we have $\mathrm{CJ}(P^{\bullet})$ is a resolution for $\mathrm{CJ}(\pi)$. Moreover, 
    \[ \mathrm{CJ}_{N}(\U(\fkn_B))\simeq \U[[\fkn]]\otimes \U(\fkn_L)
    \]
    as $\fkn$-module since $\U(\fkn_B)\simeq \U(\fkn)\otimes \U(\fkn_L)$. Therefore, 
    \begin{equation*}
        \rmh_i(\fkn,\mathrm{CJ}_{N}(\U(\fkn_B)))\simeq\begin{cases}
            \U(\fkn_L) & i=0\\
            0 & i>0.
        \end{cases}
    \end{equation*}
    Consequently, $\mathrm{CJ}(P^{\bullet})$ is also acyclic and \[
    \rmh_0(\fkn, \mathrm{CJ}(P^{\bullet}))\simeq \rmh_0(\fkn,P^{\bullet}),
    \]
    which implies the result.
\end{proof}
Now, we prove the comparison~\eqref{com-in-CJ eq}. We first construct the functorial map. Let $V\in \HC(\fkg,K_L)$. Then by~\eqref{glob-iso}, we have a natural embedding $V\to \widehat{V}$ as $(\fkg,K_L)$-modules, which will induce a $(\fkl,K_L)$-map for any integer $i$
\[ \rmh_i(\fkn,V)\lra \rmh_i(\fkn,\widehat{V}).
\]
Recall that $\rmh_i(\fkn,\widehat{V})$ is an Casselman-Wallach representation, see~\cite[Proposition 2.26]{WZ}. Therefore, by the Casselman-Wallach globalization, we obtain the natural map 
\begin{equation}\label{nautral comparison map}
     \rmh_i(\fkn,V)^{\infty}\lra \rmh_i(\fkn,\widehat{V}) \text{ for any integer } i.
\end{equation}

\begin{theorem}\label{complet-com}
    Let $V$ be an irreducible object in $\HC(\fkg,K_L)$. Then the natural map~\eqref{nautral comparison map} is an isomorphism.
\end{theorem}
\begin{proof}
    By the classification of the irreducible objects, there exists an irreducible $(\fkl,K_L)$-module $\tau$ such that $V\simeq \CL(\tau)$. Consider the short exact sequence
    \[
    0\lra \iota \lra  \CV(\tau)\lra \CL(\tau)\lra 0.
    \]
    To prove the comparison map~\eqref{nautral comparison map} is an isomorphism, it suffices to prove that the comparison map for $\iota$ and $\CV(\tau)$ is an isomorphism. Note that
    \[
    \min \mathrm{wt}(\iota)>\mathrm{wt}(\tau).
    \]
    Applying similar argument as above to the composition factors of $\iota$, and by~\cite[Lemma 2.26]{WZ}, after finite steps, we reduce to proving that the comparison map for any formal Verma module is an isomorphism. As a $\fkn$-module, we have
    \[
    \CV(\tau)\simeq \U[[\fkn]]\otimes \tau \quad \text{ and }\quad  \widehat{\CV(\tau)} \simeq \CV(\tau^{\infty})\simeq \U[[\fkn]]\widehat{\otimes} \tau^{\infty}.
    \]
    Consequently,
    \[
    \rmh_0(\fkn,\CV(\tau))^{\infty}\simeq \rmh_0(\fkn,\widehat{\CV(\tau)})\simeq \tau^{\infty},
    \]
    and 
    \[
     \rmh_i(\fkn,\CV(\tau))^{\infty}\simeq \rmh_i(\fkn,\widehat{\CV(\tau)})=0
    \]
    for $i>0$.
\end{proof}

\begin{proof}[Proof of the comparison conjecture]
    \begin{equation}\label{iso comp eq}
            \rmh_i(\fkn,\pi_K)^{\infty}\simeq \rmh_i(\fkn,\mathrm{CJ}(\pi))^{\infty}\simeq \rmh_i(\fkn, \widehat{\mathrm{CJ}(\pi)})\simeq \rmh_i(\fkn,\mathrm{CJ}^{\infty}(\pi))\simeq \rmh_i(\fkn,\pi),
    \end{equation}
    where the first isomorphism is Proposition~\ref{HC-com}, the second isomorphism is Theorem~\ref{complet-com}, the third isomorphism is Proposition~\ref{commutative dia prop} and the last isomorphism is~\eqref{CW-com-prove}.
\end{proof}
\begin{remark}
    When $\rmh_i(\fkn,\pi)$ is known to be Casselman-Wallach, then by the uniqueness of Casselman-Wallach globalization, there is a natural map 
    \[
    \rmh_i(\fkn,\pi_K) ^{\infty}\lra \rmh_i(\fkn,\pi)
    \]
    induced by $\rmh_i(\fkn,\pi_K)\to \rmh_i(\fkn,\pi)$. One can verify that it is identical to the composition~\eqref{iso comp eq}.
\end{remark}
Lastly, we remark that the maximal parabolic subgroups of classical groups have finitely many coadjoint orbits on their unipotent radicals, see~\cite[Theorem 1.1]{HG99}. Hence, to complete our strategy, it is desirable to have following lemma.
\begin{lemma}\label{reduce to max}
   Let $P\subset Q$ be two standard parabolic subgroups of a real reductive group $G$, with Levi decompositions $P=LN$ and $Q=MU$. If the comparison conjecture~\ref{comparison conj} holds for Casselman-Wallach representations of $G$ with parabolic subgroup $Q$, and Casselman-Wallach representations of $M$ with parabolic subgroup $M\cap P$. Then, the comparison conjecture~\ref{comparison conj} also holds for Casselman-Wallach representations of $G$ with parabolic subgroup $P$.
\end{lemma}
\begin{proof}
    The proof is similar as~\cite[Lemma 8.11]{WZ}. Let $\pi$ be a Casselman-Wallach representation of $G$. Let $V:=N \cap M$. Consider the double complex given by the Koszul resolution
    \[
    P_{p,q}:=\wedge^p \fkv\otimes \wedge^q \fku\otimes \pi,
    \]
    then 
    \[
    \rmh_i(\Tot(P_{\bullet,\bullet}))=\rmh_i(\fkn,\pi).
    \]
    The total complex admits a finite increasing filtration $\CF^j:=\Tot_{p\leq j,\bullet}$ with 
    \[
       E_2^{p,q}=\rmh_p(\fkv,\rmh_q(\fku,\pi)).
    \]
    Similarly, since the Casselman-Wallach globalization functor is exact, there is a spectral sequence for $\rmh_i(\fkn,\pi_K)^{\infty}$ whose
    \[
    E_2^{p,q}=\rmh_p(\fkv,\rmh_q(\fku,\pi_K))^{\infty}.
    \]
    Then by assumption, the natural isomorphism map
    \[
    \rmh_p(\fkv,\rmh_q(\fku,\pi_K))^{\infty}\simeq \rmh_p(\fkv,\rmh_q(\fku,\pi)_K)^{\infty}\simeq \rmh_p(\fkv,\rmh_q(\fku,\pi))
    \]
    intertwines two spectral sequences, where the first isomorphism uses the comparison for $Q$, and the second uses the comparison for $M$ with parabolic $M\cap P$. This implies that the natural map $\rmh_i(\fkn,\pi_K)^{\infty}\to\rmh_i(\fkn,\pi)$ is an isomorphism for any integer $i$.
\end{proof}

\section{Transitivity of (co)-standard objects}\label{trans sec}
In this section, we investigate the transitivity of the Casselman-Jacquet functor. Our idea is to reduce the computation of the Casselman-Jacquet functor to maximal parabolic cases. We remark that, by Theorem~\ref{unique sub thm} and the related discussion in~\cite[subsection 2.5]{WZ}, the three families
\[
\{[\CV(\tau)]\},\qquad \{[\CQ(\tau)]\},\qquad \{[\CL(\tau)]\},
\]
where $\tau$ runs over irreducible objects in $\mathrm{CW}_L$, each forms a basis for the Grothendieck group of the category $\CC(\fkg,L)_f$. Our idea is to study the relationship between the smooth Casselman-Jacquet functor and these bases.

We first introduce the viewpoint from the dual side. Let $P=LN$ be an almost linear Nash group with a Levi decomposition. Let $\pi\in\Smod_P$. By~\cite[Proposition 5.1]{CWYZ}, the canonical map $\pi \to \mathrm{CJ}^{\infty}(\pi)$ is surjective. Thus, we have a topological isomorphism
\begin{equation}\label{dual iso}
      \mathrm{CJ}_N^{\infty}(\pi)'\simeq \varinjlim_k  \, (\pi')^{\fkn^k},
\end{equation}
where the right-hand side is equipped with the direct limit topology, which is equivalent to the subspace topology under this circumstance. On the other hand, if we assume moreover $\pi$ is nuclear, then
  \begin{equation}\label{CJ dual side}
              \mathrm{CJ}^{\infty}_{N}(\pi)\simeq  (  \varinjlim_k \, (\pi')^{\fkn^k})'
  \end{equation}
as $(\fkg,P)$-modules since the nuclear Fr\'echet space is reflexive.
We recall two properties of the nuclear Fr\'echet space.
\begin{enumerate}
    \item The closed subspace and Hausdorff quotient space of a nuclear Fr\'echet space is still nuclear F\'echet.
    \item The projective limit of a family of nuclear Fr\'echet spaces is still nuclear Fr\'echet.
\end{enumerate}

Let $G$ be a real reductive group. Let $P\subset Q$ be two standard parabolic subgroups, with Levi decompositions $P=LN$ and $Q=MU$. Denote $V:= M\cap N$.
\begin{lemma}\label{transitivity lem}
    Let $\pi$ be a nuclear $(\fkg,Q)$-module. Then we have a natural topological isomorphism
    \[
    \mathrm{CJ}^{\infty}_{V} (\mathrm{CJ}^{\infty}_{U}(\pi))\simeq \mathrm{CJ}^{\infty}_{N}(\pi).
    \]
    In particular, as $(\fkg, P)$-modules, $\mathrm{CJ}^{\infty}_{N} (\mathrm{CJ}^{\infty}_{U}(\pi))\simeq \mathrm{CJ}^{\infty}_{N}(\pi)$.
\end{lemma}
\begin{proof}
  The isomorphism~\eqref{dual iso} implies that
  \[
 \mathrm{CJ}^{\infty}_{V} (\mathrm{CJ}^{\infty}_{U}(\pi))\simeq \left( \varinjlim_{k} \left( \varinjlim_{\ell} \, (\pi')^{\fku^{\ell}} \right)^{\fkv^k} \right)'.
    \]
    since $\mathrm{CJ}^{\infty}_{U}(\pi)$ is nuclear. Hence, it suffices to prove $\varinjlim_{k} \left( \varinjlim_{\ell} \, (\pi')^{\fku^{\ell}} \right)^{\fkv^k} = \varinjlim_{k}(\pi')^{\fkn^{k}} $. Since 
    \[
    (\pi')^{\fkn^k} \subset (\pi')^{\fku^k}\cap (\pi')^{\fkv^k},
    \]
    the right-hand side is contained in the left-hand side. Conversely, note that $\fkn$ is nilpotent, namely, there is a positive integer $N$ such that
    \begin{equation}\label{nilpotent index}
       [[\fkn, \fkn  \underbrace{ ]\dots,\fkn]}_{N}=0.  
    \end{equation}
    Thus, for any two fixed integers $k,\ell$, there exists a sufficiently large integer $t$, such that for any monomial 
    \[
    X= X_1 \dots X_t, X_i \in \fku \text{ or } X_i \in\fkv,
    \]
    either $X\in \U(\fkg)\fku^{\ell}$ or $X\in \U(\fkg) \fkv^k$. Consequently, we obtain the inverse containment.   
\end{proof}

In what follows, we discuss the relation between the Casselman-Jacquet functor and (dual) Verma modules. 
\begin{proposition}\label{trans co-std}
    Let $\tau$ be a Casselman-Wallach representation of $M$. Then we have a natural isomorphism
    \[
  \mathrm{CJ}_{N}^{\infty}(\CQ(\tau))  \simeq \CQ(\mathrm{CJ}_{V}^{\infty}(\tau)) .
    \]
\end{proposition}
\begin{proof}
    Since $\CQ(\tau)$ is nuclear, by~\eqref{CJ dual side},
    \[
    \mathrm{CJ}_{N}^{\infty}(\CQ(\tau)) \simeq \left(\varinjlim_k (\CQ(\tau)')^{\fkn^k} \right)'\simeq \left(\varinjlim_k (\U(\fkg)\otimes_{\U(\fkq)}\tau')^{\fkn^k} \right)'.
    \]
    Hence, the result follows from the following claim:
    \begin{equation}\label{claim eq}
            \varinjlim_k (\U(\fkg)\otimes_{\U(\fkq)}\tau')^{\fkn^k}  = \U(\fkg) \otimes_{\U(\fkq)}\varinjlim_k(\tau')^{\fkv^k}
    \end{equation}
    since $\mathrm{CJ}_V^{\infty}(\tau)'\simeq\varinjlim_k(\tau')^{\fkv^k} $ for Casselman-Wallach representation $\tau$. 
    \begin{itemize}
        \item ``RHS $\subset$ LHS": For any $X\in\U(\fkg)$ and any positive integer $k$, there exists a positive integer $N$ such that $\fkn^{N} X\subset \U(\fkg) \fkn^k$.
        \item ``LHS $\subset$ RHS": Fix a set of root vectors in $\ov{\fku}$. This gives a set of standard PBW basis $\CB_{\fku}$ in $\U(\ov{\fku})$. Since $\U(\fkg)\otimes_{\U(\fkq)}\tau'\simeq \U(\ov{\fku})\otimes \tau'$, we express an element $X$ in $\varinjlim_k (\U(\fkg)\otimes_{\U(\fkq)}\tau')^{\fkn^k}$ under the basis $\CB_{\fku}$:
        \[
     X=   \sum_{B_i\in \CB_{\fku}} B_i \otimes \beta_i, \beta_i\in \tau'.
        \]
        We argue by induction on the number of nonzero coefficients $\beta_i$. Let $B_j$ be a minimal element in $\CB_{\fku}$ under the partial order defined by positive roots such that $\beta_j \neq 0$. Suppose $\fkn^k \cdot X=0$. Then for any $Y\in \fkn^k$,
        \[
        YX= \sum_{B_i\in \CB_{\fku}} B_i \otimes Y\beta_i  +\sum_{B_i\in \CB_{\fku}} B_i \otimes \gamma_i,
        \]
        where $\gamma_i\neq 0$ only if $B_i > B_{\ell}$ for some $\ell$ such that $\beta_{\ell}\neq 0$. Therefore,
        \[
        Y\beta_j =0 \text{ for any }Y\in\fkn^k,
        \]
        which implies $\beta_j \in (\tau')^{\fkv^k}$ since $\U(\fkn)\simeq \U(\fkn)\fku \oplus \U(\fkv)$. Thus, 
        \[
        X- B_j \otimes \beta_j \in \varinjlim_k (\U(\fkg)\otimes_{\U(\fkq)}\tau')^{\fkn^k},
        \]
    and the result follows from the induction hypothesis.
    \end{itemize}
\end{proof}

\begin{proposition}\label{trans std}
        Let $\tau$ be a Casselman-Wallach representation of $M$. Then we have a natural isomorphism
    \[
  \mathrm{CJ}_{N}^{\infty}(\CV(\tau))  \simeq \CV(\mathrm{CJ}_{V}^{\infty}(\tau)) .
    \]
\end{proposition}
\begin{proof}
    By the definition of the formal Verma module, the $\fku$-action on $\CV(\tau)$ can be integrated, via the logarithmic map, to a smooth moderate growth $U$-action. Hence, by Lemma~\ref{transitivity lem}, we have a short exact sequence
    \[
    0 \lra \bigcap_k \ov{\fkv^k\CV(\tau) } \lra \CV(\tau) \lra \mathrm{CJ}_{V}^{\infty}(\CV(\tau)) \simeq \mathrm{CJ}_{N}^{\infty}(\CV(\tau))\lra 0.
    \]
    On the other hand, since the functor $\CV(\cdot)$ is exact, we have another short exact sequence
    \[
    0 \lra \CV(\cap_k \ov{\fkv^k \tau}) \lra \CV(\tau) \lra \CV(\mathrm{CJ}_{V}^{\infty}(\tau))\lra 0.
    \]
    Therefore, it suffices to show that 
    \[
\bigcap_k \ov{\fkv^k\CV(\tau) }   =  \CV(\cap_k \ov{\fkv^k \tau}).
    \]
    We realize $\CV(\tau)$ as $\U[[\fku]] \widehat{\otimes} \tau $. 
    \begin{itemize}
        \item`` RHS $\subset$ LHS": Since $\CV(\tau)=\ov{\U[[\fku]]\otimes \tau}$, it suffices to prove that for any positive integer $k$, there is an integer $t$, such that
        \[
        \U[[\fku]]\otimes \fkv^t \tau \subset \fkn^k (\U[[\fku]]\otimes \tau).
        \]
        Assume that $\fkn$ satisfies~\eqref{nilpotent index}. Then for any $X\in\U[[\fku]]$, 
        \[
         [[X, \fkv\underbrace{]\dots,\fkv]}_{N}  \subset \fku^2 \U[[\fku]].
        \]
        Thus, $t=k(1+N)$ is one of the choices.
        \item `` LHS $\subset$ RHS": Let $x\in \bigcap_k \ov{\fkv^k\CV(\tau) }$, then there exists a sequence of elements $\{x_i\}$ such that
        \[
        x_i\in \fkv^k(\U[[\fku]]\otimes \tau) \text{ and } x_i \to x.
        \]
        Since $\fkn$ is nilpotent, for any integer $k$, there exists a sufficiently large positive integer $t(k)$ such that
        \[
        \fkv^t \U[[\fku]] \subset \U[[\fku]] \fku^k + \U[[\fku]] \fkv^k .
        \]
        Consequently, $x_{t(k)}= y_{t(k)} +z_{t(k)}$, for some
        \[
        y_{t(k)} \in \U[[\fku]]\fku^k \otimes \tau \quad, \quad x_{t(k)}\in \U[[\fku]] \otimes \fkv^k \tau.
        \]
        In other words, $x_{t(k)}$ converges to $x$ since $y_{t(k)}$ converges to zero, which implies that $x\in \CV(\ov{\fkv^k\tau})$ for any positive integer $k$.
    \end{itemize}
\end{proof}

\end{document}